\documentclass[11pt,english]{article}
\usepackage{latexsym}
\usepackage{amscd}
\usepackage{amsthm}
\usepackage[utf8]{inputenc}
\usepackage{a4wide}
\usepackage{amsfonts}
\usepackage{color}
\usepackage{amssymb}
\usepackage{graphicx}
\usepackage{babel}
\usepackage{color}

\usepackage{hyperref}
\usepackage{multicol}
\usepackage[toc,page]{appendix}
\usepackage{amsmath}
\usepackage{tikz}
\usepackage{graphicx}
\usetikzlibrary{matrix}

\usepackage{mathrsfs}
\usepackage{dsfont}
\usepackage{amsthm}
\usepackage{amsbsy}
\usepackage{textcomp}
\usepackage{euscript}

\usepackage{authblk}

\usepackage{libertine}
\usepackage[T1]{fontenc}
\DeclareMathOperator{\supp}{supp}
\DeclareMathOperator{\diam}{diam}

\newcommand{\br}{\operatorname{br}}
\newcommand{\vol}{\operatorname{vol}}

\def\Xint#1{\mathchoice
{\XXint\displaystyle\textstyle{#1}}%
{\XXint\textstyle\scriptstyle{#1}}%
{\XXint\scriptstyle\scriptscriptstyle{#1}}%
{\XXint\scriptscriptstyle\scriptscriptstyle{#1}}%
\!\int}
\def\XXint#1#2#3{{\setbox0=\hbox{$#1{#2#3}{\int}$ }
\vcenter{\hbox{$#2#3$ }}\kern-.6\wd0}}

\def\dashint{\Xint-}

\theoremstyle{plain}\newtheorem{Th}{Theorem}
\theoremstyle{plain}
\newtheorem{Def}[Th]{Definition}

\newtheorem{Rem}[Th]{Remark}

\theoremstyle{plain}
\theoremstyle{plain}
\theoremstyle{plain}\newtheorem{Le}[Th]{Lemma}

\renewcommand{\leq}{\leqslant}
\renewcommand{\geq}{\geqslant}

\definecolor{chn}{rgb}{0, 0, 0}

\newcommand{\Rr}{\mathbb{R}}
\newcommand{\Tt}{\mathbb{T}}

\newcommand{\Zz}{\mathbb{Z}}
\newcommand{\Nn}{\mathbb{N}}

\date{}

\title{Trace operator on von Koch's snowflake}
\author[1,2]{Krystian Kazaniecki$^{\dagger,}$}
\author[3]{Micha{\l} Wojciechowski}
\affil[1]{Institute of Analysis, Johannes Kepler University Linz}
\affil[2]{Institute of Mathematics, University of Warsaw}
\affil[3]{Institute of Mathematics, Polish Academy of Sciences}
\begin{document}
\maketitle
\begin{abstract}
We study properties of the boundary trace operator on the Sobolev space $W^1_1(\Omega)$. Using the density result by Koskela and Zhang \cite{MR3519964}, we define a surjective operator \mbox{$Tr: W^1_1(\Omega_K)\rightarrow X(\Omega_K)$}, where $\Omega_K$ is von Koch's snowflake and $X(\Omega_K)$ is a trace space with the quotient norm. Since $\Omega_K$ is a uniform domain whose boundary is Ahlfors-regular with an exponent strictly bigger than one, it was shown by L. Mal\'y \cite{Malysb} that there exists a right inverse to $Tr$, i.e. a linear operator $S: X(\Omega_K) \rightarrow W^1_1(\Omega_K)$ such that $Tr \circ S= Id_{X(\Omega_K)}$. In this paper we provide a different, purely combinatorial proof based on geometrical structure of von Koch's snowflake. Moreover we identify the isomorphism class of the trace space as $\ell_1$. As an additional consequence of our approach we obtain a simple proof of the Peetre's theorem \cite{MR0552011} about non-existence of the right inverse for domain $\Omega$ with regular boundary, which explains Banach space geometry cause for this phenomenon.
\end{abstract}

\let\thefootnote\relax\footnotetext{\!\!\!\!\!\!\!\!\!\!\textbf{$\dagger$} Corresponding author
}

It was shown by Gagliardo (\cite{MR0102739}) that the trace operator maps the space $W^{1}_{1}(\Omega)$ onto $L^1(\partial \Omega)$ for domains with Lipschitz boundary. From this theorem immediately arises a question whether there exists a right inverse operator to the trace, i.e. a continuous, linear operator $S : L^1(\partial \Omega) \rightarrow W^{1}_{1}(\Omega)$ such that $Tr\circ S= Id$. It turns out that in general such operator does not exist. This was proved by Peetre (\cite{MR0552011}). In his paper he has shown the non-existence of right inverse to the trace operator for a half plane. From that by straightening out the boundary one can deduce the non-existence of the right inverse for $\Omega$ with a Lipschitz boundary. More recent proofs can be found in \cite{MR1979184}, \cite{MR2882877}. Note that the trace was studied also in the context of spaces given by more general differential constraints \cite{Gmeineder2019}. In this article we present an exceptionally simple proof based on geometry of a Whitney covering and basic properties of classical Banach spaces.
\begin{Th}\label{Skladak}
Let $\Omega$ be an open domain with Lipschitz boundary and $\partial\Omega$ be a Jordan curve. Let\; $Tr:W^{1}_{1}(\Omega )\rightarrow L^1(\partial\Omega)$ be a trace operator. Then there is no continuous, linear operator  $S:L^1(\partial\Omega) \rightarrow W^{1}_{1}(\Omega)$ such that $Tr\circ S= Id_{L^1(\partial\Omega)}$. 
\end{Th} 
Actually the proof of Theorem \ref{Skladak} presented here works with weaker assumptions. Whenever 
\begin{enumerate}
    \item one can define reasonable trace operator i.e. functions which are continuous up to the boundary are dense in $W^1_1(\Omega)$ and the Trace is just the restriction operator on such functions; when $\partial\Omega$ is a Jordan curve this is provided by Koskela, Zhang theorem ( \cite{MR3519964}),
    \item trace space $X$ contains isomorphic copy of $L^1$.
\end{enumerate}

In \cite{MR1428124} Hajlasz and Martio studied the existence of a right inverse to \textcolor{chn}{the} trace operator in the case of Sobolev spaces $W^p_1(\Omega)$ for $p>1$ and they characterize trace space as a generalized Sobolev space. In paper \cite{bonk2018} the trace of Haj{\l}asz-Sobolev spaces to porous Ahlfors regular closed subspace is studied for sufficiently large exponent. 

In the case of $p=1$ the behavior of the trace space ( the smallest Banach space for which the trace operator is bounded)  changes dramatically for the domains with fractal boundary.  L. Mal\'y \cite{Malysb} characterized the trace spaces for Ahlfors regular, uniform domains in terms of Besov spaces on fractal sets and he constructed the linear extension operator from the trace space. 

The second goal of this paper is to study the trace operator on the Sobolev space on von Koch's snowflake $\Omega_K$. This is a very particular domain and the proofs provided in this article depend heavily on its combinatorial structure. However they are quite different than Mal\'y approach and does not use Besov norm at all. The Besov space obtained by Mal\'y as a trace space in the case of von Koch's snowflake turns out to be isomorphic to $\ell_1$. Indeed, it follows from Theorem 1 and Theorem 3 of Chapter VI of \cite{jonsson} that the trace space is complemented subspace of $B^1_{1,1}(\mathbb{R}^2)$. By Proposition 7  p. 200 of \cite{Meyer} the space $B^1_{1,1}(\mathbb{R}^2)$  is isomorphic to $\ell_1$. By Pe{\l}czy\'nski's theorem complemented subspaces of $\ell_1$ are isomorphic to $\ell_1$ \cite{pelcz}.
Our proof identifies the trace space directly as an Arens-Eells space (otherwise known as Lipschitz free space) for a suitable metric on the boundary by a combinatorial argument. It follows by an old Ciesielski's argument  \cite{MR0132389} that this space is isomorphic to $\ell_1$. The existence of a right inverse operator is just a property of $\ell_1$. In the proof we use the structure of a specific Whitney covering of $\Omega_K$ described in the Appendix, which is quite interesting in itself. To summarize:
\begin{Th}\label{twierdzenieKocha}
Let  $Tr: W^{1}_{1}(\Omega_K) \rightarrow X(\Omega_K)$ be a trace operator, where $X(\Omega_K)$ is a trace space \textcolor{chn}{defined by} \eqref{przslad}. Then $X(\Omega_K)\simeq \ell_1$  and there exists a continuous, linear operator \mbox{$S:X(\Omega_K)\rightarrow W^1_1(\Omega_K)$} such that $Tr\circ S=Id_{X(\Omega_K)}$. 
\end{Th}
    We want to stress that the novelty of this article lies in proposing new combinatorial method of proof rather than the result itself, which is already known. To keep the clarity of the presentation we decided to focus on the von Koch's snowflake. Our method could be applied to more general class of domains (some of them could be obtained e.g. by Carleson's construction \cite{carlfrac}). However we do not know how wide the actual range of possible applications of this approach is. 
    
    In the following section we define the trace operator, trace space and auxiliary properties \textcolor{chn}{of} $BV(\Omega)$ needed in the proof.

    \vspace{5mm}


\section{Properties of $BV(\Omega)$ and trace operator}
From now on we assume that $\Omega\subset \Rr^2$, $\partial \Omega$ is a Jordan curve. Our approach to Theorem \ref{Skladak} up to technical differences works in higher dimensions. However in the proof of the Theorem \ref{twierdzenieKocha} the properties of two dimensional euclidean space are crucial. We define the trace operator and the trace space for $W^{1}_1(\Omega)$. Let us recall a notion of (slightly generalized) Whitney covering of $\Omega$.  
\begin{Def}\label{defwhit} We call the family of polygons $\mathcal{A}$ a Whitney decomposition of an open set\; $\Omega\subset \Rr^2$ if it satisfies:
\begin{enumerate}
\item\label{combilip} For $A\in\mathcal{A}$ the boundaries $\partial A$ are uniformly bi-lipschitz \textcolor{chn}{images of $\mathbb{S}^1$.}
\item $\bigcup_{Q\in\mathcal{A}} Q= \Omega$ and elements of $\mathcal{A}$ have pairwise disjoint interiors.
\item \label{compdistvol} $C^{-1}  \vol_2 A\leq dist(A, \partial \Omega)^n\leq C \vol_2 (A).$
\item \label{comneigh} If $\partial A\cap\partial B $ ha a positive one dimensional Hausdorff measure then
\begin{enumerate}
\item $C^{-1} \leq \frac{\vol_2(A)}{\vol_2(B)}\leq C$. 
\item $C^{-1} \leq \frac{l(\partial A)}{l(\partial B)}\leq C.$
\item $C^{-1}l(\partial A)\leq l(\partial A\cap \partial B)\leq C^{-1}l(\partial A)$, 
\end{enumerate}
 where $l(\cdot)$ denotes length of a curve, and $\vol_2$ denotes the area of the polygon.
\item For a given polygon $A\in\mathcal{A}$ there exists at most N polygons $B\in\mathcal{A}$ s.t. $\partial A\cap\partial B\neq \emptyset$.
\end{enumerate} 
For the purpose of this article we will also assume that polygons of $\mathcal{A}$ are uniformly star shaped in the following sense
\begin{enumerate}
\item[6.] For every $A\in\mathcal{A}$ there exists a point $x\in A$ and positive numbers $\lambda$, $\tau$ s.t. $B(x,\lambda)\subset A\subset B(x,\tau)$, where $\frac{\lambda}{\tau}$ is fixed and the polygon $A$ is star shaped with respect to $x$. We call such point a center of $A$.
\end{enumerate} 
\end{Def}
Let $\mathcal{A}$ be such covering then we can define a graph describing it's geometry.
\begin{Def}
Let $\mathcal{A}$ be a Whitney decomposition. We call a graph $G:=G(\mathcal{A})=(V(\mathcal{A}),E(\mathcal{A}))=:(V,E)$ a graph of $\mathcal{A}$ if\; $V:=\mathcal{A}$ and $\{A,B\}\in E$ only if boundaries of $A$ and $B$ have intersection of positive one dimensional Hausdorff measure.
\end{Def}
We denote by $BV(\Omega)$ a space of measures of bounded variation i.e.
\[
BV(\Omega)=\{f\in L^1: \nabla f \in M(\Omega,\mathbb{R}^2)-\mbox{space of vector valued bounded borel measures}\}
\] 
with a norm 
\[
\|f\|_{BV(\Omega)}=\|f\|_{L^1(\Omega)}+\|\nabla f\|_{M(\Omega, \mathbb{R}^2)}
\]
We introduce \textcolor{chn}{some} special subspaces of $BV(\Omega)$.
\begin{Def}
Let $\mathcal{A}$ be a Whitney decomposition of\; $\Omega$. We define the following subspaces of $BV(\Omega)$
\[
BV_{\mathcal{A},0}= \{F\in BV(\Omega): \forall \,A\in\mathcal{A}\;\; \int_{A} F(x)dx=0\}
\]
and
\[
BV_{G}=\{f \in BV(\Omega): \forall\, A\in \mathcal{A}\quad f|_{A}=f_{A}\in \Rr\}
\]
\end{Def}
It is a known fact that for a given Whitney decomposition the space $BV_{\mathcal{A},0}$ is a complemented subspace of $BV(\Omega)$. A proof of this fact can be found in (\cite{RW},\cite{Derezin}). 
\begin{Le}
For any domain $\Omega$: 
\[ BV(\Omega) = BV_{\mathcal{A},0} \oplus BV_{G}.
\]
\end{Le}
Let us observe that we can easily calculate the norm of function $f\in BV_{G}$.
\[
\|f\|_{BV_{G}}:=\| f\|_{BV(\Omega)} \textcolor{chn}{=} \sum_{A\in V} |f_A|\;\vol_2(A)+ \sum_{\{A,B\}\in E} |f_{A}-f_{B}|\; l(\partial A\cap \partial B)
\]
\begin{Def}\label{dobredrzewo}
Let \;$\Omega$ be \textcolor{chn}{a} simply connected planar domain with Poincare inequality for $p=1$ and \textcolor{chn}{$\mathcal{A}$} be its Whitney decomposition. We will call a spanning tree $T=(V_T,E_T)$ of the graph $G(\mathcal{A})$ a Whitney tree of $\mathcal{A}$ if it satisfies \textcolor{chn}{the} following conditions: 
\begin{enumerate}
\item for every $f\in BV_{G}(\Omega)$
\begin{equation}\label{normaBVG}
\|f\|_{BV_{G}}\simeq \|f\|_{BV_T}:=\sum_{A\in V_T} |f_A|\; \vol_2(A)+\sum_{\{A,B\}\in E_T} |f_{A}-f_{B}| \;l(\partial A\cap \partial B),
\end{equation}
\item for every point $x$ on the boundary there is \textcolor{chn}{an} infinite branch $\br(x)$ of \;$T$ s.t. $ \br(x)\cong \Zz_{+}$ and $\operatorname{dist}(A_n,x)\rightarrow 0$ as $n\rightarrow \infty$, where $A_n\in\br(x)$. For a sequence of real numbers $\{a_{A_n}\}$ we call a limit $\lim_{n\rightarrow \infty} a_{A_n}$ a limit along the branch $\br(x)$.
\end{enumerate}
\end{Def}
In their unpublished preprint Derezinski, Nazarov, Wojciechowski \cite{DNW} have proven that for any bounded simply connected planar domain there exists a Whitney tree of the graph $G(\mathcal{A})$. However in \textcolor{chn}{the present} paper we will not use this result. An explicit construction of \textcolor{chn}{a} Whitney tree for von Koch's snowflake is given in Section \ref{kochowanie} and in the Appendix. Moreover \textcolor{chn}{the} obtained tree has very regular structure.\\ 

It follows immediately  that $BV_G\cong BV_T$, where $BV_T$ is a set $BV_G$ with the norm \mbox{$\|\cdot\|_{BV_T}$}. 
Using the above notation we define trace of $f\in W^{1}_{1}(\Omega)$. Since $\Omega$ is a domain with a Jordan curve as boundary it follows from \textcolor{chn}{theorem of Koskela and Zhang } ( \cite{MR3519964}, see \cite{Koskela2017} for the case d>2) that restrictions of Lipschitz function $\operatorname{Lip}(\Rr^2)$ are dense in $W^{1}_{1}(\Omega)$. For $f\in C(\overline{\Omega})\cap W^{1}_{1}(\Omega)$ we define the trace operator as a restriction of $f$ to the boundary. We define a trace space $X(\Omega)$ as completion of a space $Tr(C(\overline{\Omega})\cap W^{1}_{1}(\Omega))$ with respect to the norm $\|\cdot\|_X$, where 
\begin{equation}\label{przslad}
\|g\|_{X(\Omega)}:= \inf\{\|f\|_{W^{1}_{1}(\Omega)}: Tr f= g \mbox{ and } f\in C(\overline{\Omega})\cap W^{1}_{1}(\Omega) \}.
\end{equation}
Since Lipschitz functions on $\Omega$ are dense in $W^{1}_{1}(\Omega)$ we can define \textcolor{chn}{the} trace operator on \textcolor{chn}{the} whole space $W^{1}_{1}(\Omega)$. It is obvious that $Tr: W^{1}_{1}(\Omega) \rightarrow X(\Omega)$ is a continuous linear operator and it is surjective. We want to extend the trace operator to $BV(\Omega)$ (see also \cite{Lahti}).
\begin{Le}\label{liniowBV}
There exists a continuous, linear operator $\Phi : BV_G \rightarrow W^{1}_{1}(\Omega)$ s.t. for every $A \in \mathcal{A} $
\begin{equation}\label{coscostam}
f_A=\dashint_{A} f(y) dy =\dashint_{A} \Phi(f)(y) dy + o(\operatorname{dist}(A,\partial \Omega)).
\end{equation}
\end{Le}
\begin{proof}
\textcolor{chn}{
We begin with a construction of an operator which smooths out the function from $BV(\Omega)$ inside of a fixed ball $B\subset\Omega$. We formulate it in two-dimensional case, however it is valid in domains of arbitrary dimension with the analogous proof.
\begin{Le}\label{ballsmooth} For every ball $B\subset \Omega\subset \mathbb{R}^2$ there exists a bounded linear operator $T_B : BV(\Omega)\rightarrow BV(\Omega)$ such that $T_B f|_{(\Omega\backslash B)} = f|_{(\Omega\backslash B)}$ and $(T_Bf)|_B \in W^{1}_{1}(B)$ for $f \in BV(\Rr^d)$. Moreover
\[
\|(T_Bf)|_B\|_{W^{1}_{1}(B)}\leq K \|f|_B\|_{BV(B)}
\]
with a constant independent of the radius of a ball. Moreover for $f \in BV(\Rr^d)\cap L^\infty(B)$ we have 
\begin{equation}\label{suponball}
\|(T_Bf)|_B\|_{L^{\infty}(B)}\leq C \|f|_B\|_{L^{\infty}(B)}. 
\end{equation}
\end{Le}
\begin{proof}[Proof of Lemma \ref{ballsmooth}]
\textcolor{chn}{ We will use operator $T_1$ from Proposition 4.2 in \cite{MR1979184}. It is a continuous operator from $BV(\Rr^d_+)$ to $W^1_1(R^d_+)$ satisfying  $Tr T_1 f= Tr f$. This operator is given by a convolution like formula: 
\[
T_1 f (x)=\int_{R^{d}_{+}}  \phi\left(\frac{y-y^{'}}{x_d},\frac{x_d^{'}}{x_d}\right) x_d^{-d} f(y^{'},x_{d}^{'}) dy^{'}dx_{d}^{'}, \qquad\mbox{ for } x=(y,x_d)\in\Rr^{d}_+.
\]
Therefore for $f\in BV(\Rr^d_+)\cap L^{\infty}$ we have
\[
\|T_1 f\|_{L^{\infty}}\leq \|f\|_{L^{\infty}}.
\]
For every $x\in \partial B_0:=B(0,1)$ we have $U_x=B(x,r_x)\cap B_0$ and diffeomorphism $p_x: U_x \rightarrow \Rr^d_+$ such that $p_x(U_x\cap \partial B_0)\subset R^{d-1}\times\{0\}.$ There exists $V_x=B(x,\tilde{r}_x)\cap B_0$ such that for $f\in BV(\Rr^d_+)$ we have 
\[
\supp f\subset p_x(V_x)\Rightarrow \supp T_1f \subset p_X(U_x).
\]
We can choose a finite cover of $B_0$ by $\{V_{x_j}\}_{j=1}^N$ and $B(0,r_0)$ with $r_0<1$. There exists a decomposition of unity corresponding to this covering. We denote it by $\phi_0, \phi_{x_1},\ldots, \phi_{x_N}$. We define our operator by the formula
\[
H f(x)= (\phi_0 f)*\psi (x)+\sum_{j=1}^N (T_1 ((\phi_{x_j} f)\circ p^{-1}_{x_j}))\circ p_{x_j} (x) ,
\]
where $\psi$ is a mollifier with $\supp \psi \subset B(0,\epsilon)$, $r+ 2 \epsilon<1$.
Clearly
\[
Tr_{\partial B_0} (\phi_0 f)*\psi (x) =Tr_{\partial B_0}\phi_0 f=0.
\]
From properties of the operator $T_1$ we have 
\[
Tr_{\partial B_0} (T_1 ((\phi_{x_j} f)\circ p^{-1}_{x_j}))\circ p_{x_j}= Tr_{\partial B_0} (\phi_{x_j} f) 
\]
for every $j=1,...,N$.
Thus
\[
Tr_{\partial B_0} H f(x)= Tr_{\partial B_0}\phi_0 f + \sum_{j=1}^N Tr_{\partial B_0} (\phi_{x_j} f)  = Tr_{\partial B_0} f.
\]
Moreover from the properties of the operator $T_1$ and convolution for $f\in BV(B_0)\cap L^{\infty}$  we get 
\[
\|H f\|_{L^{\infty}}\leq \|f\|_{L^{\infty}}.
\]
Let $J_{s}$ be a homothety with a center at zero and scale $s$ we define 
\[
H_{B(0,s)}f =  J_s\circ H\circ J_{\frac{1}{s}}  (f-\dashint_{B(0,s)} f)+\dashint_{B(0,s)} f
\]
First we estimate the norm of the gradient 
\[
\begin{split}
\int_{B(0,s)} |\nabla J_s\circ H\circ J_{\frac{1}{s}}  &(f-\dashint_{B(0,s)} f)|\leq \int_{B(0,1)} s^{-n+1} |\nabla  H\circ J_{\frac{1}{s}}  (f-\dashint_{B(0,s)} f)|  
\\ &\leq s^{-n+1} \|H\circ J_{\frac{1}{s}}  (f-\dashint_{B(0,s)} f)\|_{BV(B(0,1))}
\\&\leq s^{-n+1} \|H\|_{BV(B_0)\rightarrow BV(B_0) } \| J_{\frac{1}{s}}  (f-\dashint_{B(0,r)} f)\|_{BV(B_0)}
\\&\stackrel{\mbox{Poincar\'e}}{\leq} C s^{-n+1} \|H\|_{BV(B_0)\rightarrow BV(B_0) } \| \nabla J_{\frac{1}{s}}  (f-\dashint_{B(0,s)} f)\|_{L^1(B(0,1))}
\\&= C \|H\|_{BV(B_0)\rightarrow BV(B_0) } \| \nabla  (f-\dashint_{B(0,s)} f)\|_{L^1(B(0,s))}
\\&= C \|H\|_{BV(B_0)\rightarrow BV(B_0)} \| \nabla  f\|_{L^1(B(0,s))}.
\end{split}
\]
Note that $T_1$ is a bounded linear operator on $L^1(R^d_+)$. Thus $H$ is bounded linear operator on $L^1(B(0,1))$. We estimate the norm of the function
\[
\begin{split}
  \|J_s\circ H\circ J_{\frac{1}{s}}  (f-\dashint_{B(0,s)} f)\|_{L^1(B(0,s))}&=s^{-n}\| H\circ J_{\frac{1}{s}}  (f-\dashint_{B(0,s)} f)\|_{L^1(B_0)}
  \\ &\leq s^{-n} \|H\|_{L^1(B_0)\rightarrow L^1(B_0) } \| J_{\frac{1}{s}} (f-\dashint_{B(0,s)} f)\|_{L^1(B_0)}
  \\& = \|H\|_{L^1(B_0)\rightarrow L^1(B_0) } \| f-\dashint_{B(0,s)} f\|_{L^1(B(0,s))}
  \\& \leq 2 \|H\|_{L^1(B_0)\rightarrow L^1(B_0) }\| f\|_{L^1(B(0,s))}
\end{split}
\]
It is clear $H_B$ satisfies $L^{\infty}$ estimates. Now we define
\[
T_B f= f \mathds{1}_{\Omega\backslash B} + H_B f \mathds{1}_{B}
\]
Note that from the properties of $H_B$ follows that $T_B$ satisfies all of the properties from the statement of the Lemma \ref{ballsmooth}.}
\end{proof}
We return to the proof of Lemma \ref{liniowBV}. We a define family of balls $\{B_{x,A}\}$ where $A\in\mathcal{A}$ and $x\in\partial A$. We put $B_{x,a}=B(x,r_A)$, where $r_A= (\operatorname{diam}A)^{4}$. This collection is a Besicovitch covering of $\bigcup_{A\in \mathcal{A}} \partial A$. There are at most $c_2$ families $\mathcal{F}_j$ of disjoint balls, which cover $\bigcup_{A\in \mathcal{A}} \partial A$. We define operators
\[
\Phi_j f= \sum_{B\in \mathcal{F}_j} T_B f
\]
and 
\[
\Phi f= \Phi_1 \circ \Phi_2 \circ \cdot\circ \Phi{c_2}\, f.
\]
It is clear, that we smooth out the function on neighbourhoods of every point $x\in\bigcup_{A\in \mathcal{A}} \partial A$. Hence for $f\in BV_G$ we get a function  $\Phi f\in W^1_1(\Omega)$. Moreover
\[
\|\Phi\|_{BV(\Omega)\rightarrow BV(\Omega)}\leq K^{c_2}.
\]
Functions $f\in BV_G$ are constant on $A\in\mathcal{A}$. Thus
\begin{equation}\label{241023}
|f_A|\leq C(\mathcal{A}) (\diam A)^{-2} \|f\|_{L^1(\Omega)},    
\end{equation}
where $C(\mathcal{A})$ depends on constants from Definition \ref{defwhit}.\ref{compdistvol}. Balls from $F_{c_2}$ only intersect the neighbouring cubes. By Definition \ref{defwhit}.\ref{compdistvol}-\ref{comneigh} we know that on the neighbouring cubes the right hand sides of \eqref{241023} are comparable. Therefore by \eqref{suponball} we get
\[
\|\Phi_{c_2} f|_A\|_{L^{\infty}}\leq C (\diam A)^{-2} \|f\|_{L^1(\Omega)}.
\]
The estimates above remain comparable on the neighbouring cubes. Repeating this argument $c_2$ times we get
\[
\|\Phi f|_A\|_{L^{\infty}}\leq C (\diam A)^{-2} \|f\|_{L^1(\Omega)}.
\]
The function is modified only on the set $P= \bigcup_{j=1}^{c_2}\bigcup_{B\in\ F_j} B$. The radii of balls intersecting $A\in \mathcal{A}$ are comparable to $(\diam A)^4$. Thus for $C$ depending on the bi-lipschitz constant (Def \ref{defwhit}.\ref{combilip})  we have 
\[
\vol (A\cap P)\leq C (\diam A)^5.
\]
Therefore 
\[
\begin{split}
 |\dashint_{A} f(y) dy - \dashint_{A} \Phi(f)(y) dy |&\leq C (\diam A)^{-2} \vol (A\cap P)(\|\Phi f|_A\|_{L^{\infty}}+|f_A| )
 \\ &\leq C (\diam A)^{-2} (\diam A)^5 (\diam A)^{-2} \|f\|_{L^1(\Omega)}= C \diam A.   
\end{split}
\]}
\end{proof}


Let $P: BV(\Omega)\rightarrow BV_G$ be a projection from $BV(\Omega)$ onto $BV_G$. We define \mbox{$\widetilde{Tr}: BV(\Omega)\rightarrow X(\Omega)$} by the formula 
\[
\widetilde{Tr} f = Tr\, \Phi\,(P f)\qquad \forall\; f\in BV(\Omega).
\] 
If $f\in C(\overline{\Omega})\cap W^{1}_{1}(\Omega)$ then the function $\Phi\, (P f)$ is continuous on $\overline{\Omega}$. Therefore its trace is a restriction of $\Phi\, (P f)$ to the boundary. However the value of the restriction at point $x\in \partial \Omega$ for the function from $C(\overline{\Omega})$ is equal to the limit of $\dashint_{A}\Phi( P f(y)) dy$ along the branch $\operatorname{br}(x)$. From \eqref{coscostam} and the definition of the space $BV_T$
\[
\Phi(P(f))(x)= f(x) \qquad \forall \,x\in\partial\Omega.
\] 
Since $f\in C(\overline{\Omega})\cap W^{1}_{1}(\Omega)$ are dense in $W^{1}_{1}(\Omega)$ and $\widetilde{Tr} f= Tr f$ the operator $\widetilde{Tr}$ is an extension of the trace operator to $BV(\Omega)$. We will abuse the notation and from now on we will denote $\widetilde{Tr}$ by $Tr$.  From the definition of \textcolor{chn}{the trace operator} it follows that
\begin{equation}\label{sladzero}
Tr f = 0 \qquad \forall \, f \in BV_{\mathcal{A},0}.
\end{equation} 
\section{Proof of Peetre's theorem}
In this section we will give a proof of Theorem \ref{Skladak}.  
\begin{proof}
Since $\Omega$ has Lipschitz boundary by theorem of Gagliardo $X(\Omega)\cong L^1(\partial \Omega)$ - \textcolor{chn}{the} space of functions integrable with respect to the $1$-dimensional Hausdorff measure. Let us denote by $P: BV(\Omega) \rightarrow  BV_{G}$ the projection onto $BV_G$. Assume there exists $S: L^{1}(\partial\Omega)\rightarrow W^{1}_1(\Omega)\subset BV(\Omega)$ such that $Tr\circ S =Id_{L^1(\partial \Omega)}$. Then the following diagram is commutative 
\begin{center}
\begin{tikzpicture}
  \matrix (m) [matrix of math nodes,row sep=3em,column sep=4em,minimum width=2em]
  {
     L^1(\partial\Omega)& BV(\Omega) & L^1(\partial \Omega) & \\
     & BV_{G} &  \\};
  \path[-stealth]
   (m-1-1) edge node [above] {S} (m-1-2)
    (m-1-2) edge node [left] {P} (m-2-2)
            edge  node [above] {$Tr$} (m-1-3)
      (m-2-2) edge node [below] {$Tr$} (m-1-3)
          ;
\end{tikzpicture}
\end{center}
 From \eqref{sladzero} and \textcolor{chn}{the theorem of Gagliardo} we conclude that $Tr|_{BV_{\textcolor{chn}{G}}}$ is onto $L^1(\partial\Omega)$. On the other hand,  $Tr\circ P\circ S = Id_{L^1(\partial \Omega)}$. Hence $L^1(\partial 
\Omega)$ is isomorphic to a subspace of $BV_G$. The definition of $BV_G$ implies that $BV_G$ is isomorphic to a subspace of $\ell^1(V)\oplus\ell^1(E)\cong \ell^1$. Since the measure on the boundary is non atomic, $L^1(\partial \Omega)\cong L^1(\mathbb{T})$. However, it is well known that $L^1$ could not be embedded in $\ell^1$. (To see this, note that by Khintchine inequality, Radamacher functions span $\ell^2$ in $L^1$ space. The space $\ell^2$ could not be embedded in $\ell^1$ because every subspace of $\ell^1$ contains a copy $\ell^1$ (\cite{MR0500056}, Proposition 1.a.11).
\end{proof}
\section{Trace operator on von Koch's snowflake}\label{kochowanie}
Let $\Omega_K$ be a domain bounded by von Koch's curve. Since $\Omega_K$ is simply connected and von Koch's curve is a Jordan curve, we can use all the properties from the first section. It is enough to show that there exists a right inverse $S: X(\Omega_K) \rightarrow BV_G$ to the trace on $BV_G$ because then $\Phi\circ S: X(\Omega_K)\rightarrow W^{1}_{1} (\Omega_K)$ and $Tr\circ  \Phi\circ S =Id_{X(\Omega_K)}$, where $\Phi$ is an operator from Lemma \ref{liniowBV}. \\
It is a well known fact that $\Omega_K$ satisfies Poincar\'e inequality (eg. \cite{MR1359964}). Therefore \textcolor{chn}{(here and further in this section the constants are absolute and specific for von Koch's snowflake and it's concrete Whitney's covering; in possible generalizations for other domains given by Carleson construction they will depend on specific geometry, bi-lipschitz constants of Whitney's covering and corresponding scaling factor).}
\[
\bigg\|f-\dashint_{\Omega_K}f(y) dy\bigg\|_{L^1(\Omega_K)}\leq \textcolor{chn}{C(\Omega_K)}|\nabla f|_{\Omega_K}, 
\]
where $|\mu|_{\Omega_K}$ is a total variation of a measure $\mu$ on $\Omega_K$. This inequality implies 
\[
BV_G\cong \dot{BV}_T \oplus \Rr,
\]
where 
\[
\dot{BV}_T= BV_{\textcolor{chn}{T}}\slash P_0, 
\]
where $P_0$ is the space of constant functions on $\Omega_K$. In this case the norm is equal to
\[
\|f\|_{\dot{BV}_T}=\sum_{\{A,B\}\in E_T} |f_{A}-f_{B}| l(\partial A\cap \partial B). 
\]
 Similarly $X(\Omega_K)= \Rr\oplus \dot{X}(\Omega_K)$ for \textcolor{chn}{the} quotient space $\dot{X}(\Omega_K)=X\slash P_0$. \textcolor{chn}{We have already established that for all $g\in \dot{X}(\Omega_K)$ we have
 \[
 \|g\|_{\dot{X}(\Omega_K)}=\inf\{\|f\|_{\dot{BV}_T}: Tr f:= g\}. 
 \] We reduce the problem to finding a right inverse operator to the trace $Tr:\dot{BV}_T \rightarrow \dot{X}(\Omega_K)$. 
We will show it's existence for a carefully chosen Whitney covering.} Construction of this covering is described in the Appendix. We introduce the following notation
\begin{Def}
For a given tree $T$ by $R:=R(T)$ we will denote the root of\; $T$. For a vertex $A\in V_T$ by $D_n(A)$ we denote descendants of $A$ of order exactly $n$ and we put $D_n=D_n(R)$. For a vertex $A\in V_T$ by $A\downarrow$ we denote its unique father. We will denote by $D\!\!\uparrow\!(A)$ the set of all descendants of $A$ i.e. $D\!\!\uparrow\!(A)=\bigcup_{n} D_n(A)$.
\end{Def} 
We take a covering $\mathcal{A_K}$ as shown on the Figure \ref{KOchsnow}.
\begin{figure}
\begin{center}
\includegraphics[width=25em]{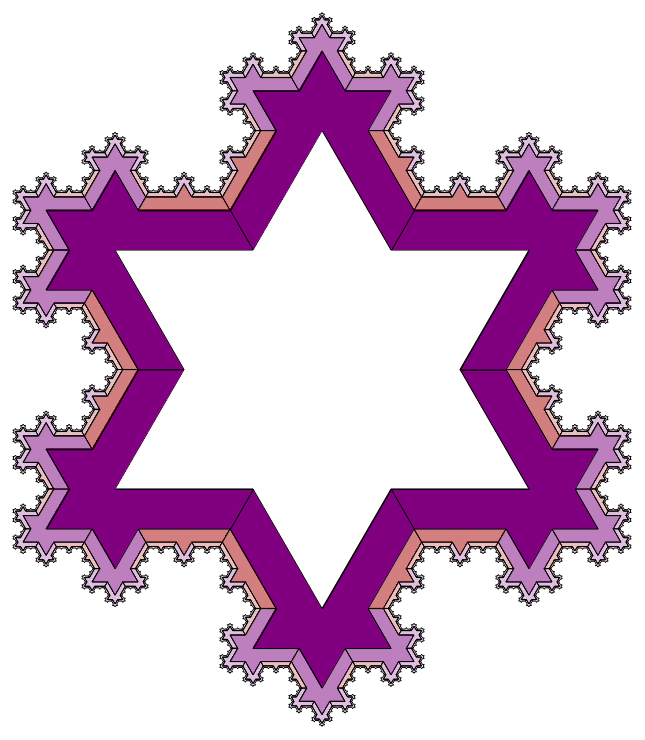}
\caption{Self similar Whitney decomposition of von Koch's snowflake}\label{KOchsnow}
\end{center}
\end{figure}
 This covering of von Koch's snowflake is easy to describe if we look at its Whitney tree $T_K$. The root of $T_K$ is a six pointed star with six "pants" shaped descendants. We denote it by $R$. In this tree there are three types of polygons/vertices. The aforementioned root, "pants" shaped polygons and "palace" shaped polygons. The type of a vertex describes direct descendants of this vertex (Figure \ref{potomkowie}).
 \begin{figure}
\begin{center}
\includegraphics[width=15em]{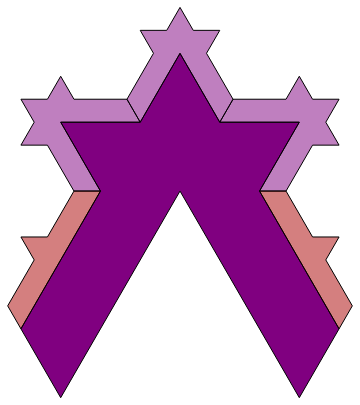}
\includegraphics[width=15em]{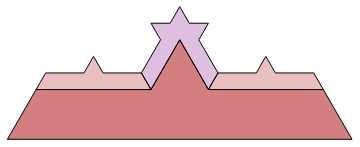}
\caption{On the left "pants" shaped polygon and its descendants, on the right "palace" shaped polygon and its descendants}\label{potomkowie}
\end{center}
\end{figure}
Polygons in $D_{n+1}$ are similar to polygons from $D_n$ with a scale $\frac{1}{3}$. The tree $T_K$ is the tree from Definition \ref{dobredrzewo}. Indeed let G be a Whitney graph of this decomposition. We denote by $H_n=\{\{A,B\}\in E_{G}: A\in D_n, B\in D_n\}$. We see that
\[
\|f\|_{\dot{BV}_G}\simeq \sum_{n}\sum_{\textcolor{chn}{\{A,B\}}\in H_n} \frac{1}{3^n} |f_A- f_B|+ \sum_{n} \sum_{A\in D_n} \frac{1}{3^n} |f_A- f_{A\downarrow}|  
\]
However for $\{A,B\} \in H_n$ we have from \textcolor{chn}{the} triangle inequality
\[
\frac{1}{3^n}|f_{A}-f_B|\leq \frac{1}{3^n}|f_A- f_{A\downarrow}|+ \frac{1}{3^n}|f_B - f_{B\downarrow}| +\frac{1}{3}\frac{1}{3^{n-1}} |f_{A\downarrow}-f_{B\downarrow}|
\]
Since we use edge $\{A,A\downarrow\}$ in the estimate for at most two edges from $H_n$ we get  
\[
\sum_{\textcolor{chn}{\{A,B\}}\in H_n} \frac{1}{3^n} |f_A- f_B|\leq \frac{1}{3} \sum_{\textcolor{chn}{\{A,B\}}\in H_{n-1}} \frac{1}{3^{n-1}} |f_A- f_B| + 2 \sum_{A\in D_n} \frac{1}{3^n} |f_A- f_{A\downarrow}|,   
\]
We get by induction
\[
\sum_{n}\sum_{\textcolor{chn}{\{A,B\}}\in H_n} \frac{1}{3^n} |f_A- f_B|\leq 2 \sum_{n}\sum_{\textcolor{chn}{m=0}}^{n}\frac{1}{3^{n-m}}\sum_{A\in D_{\textcolor{chn}{m}}} \frac{1}{3^m} |f_A- f_{A\downarrow}|\textcolor{chn}{.}
\]
The geometric sequence with quotient $\frac{1}{3}$ is convergent. Therefore for such Whitney covering the norm of $\dot{BV}_{T_K}$ satisfies
\[
\|f\|_{BV_G}\simeq  \sum_{n=1}^{\infty}\sum_{A\in D_n} |f_{A} -f_{A\downarrow}| 3^{-n}\simeq \|f\|_{\dot{BV}_{T_K}}\textcolor{chn}{.}
\] 
Further we will use \textcolor{chn}{the} above formula as a norm on $\dot{BV}_{T_K}$.
We want to study the norm on $\dot{X}(\Omega_K)$. To be precise, we want to define and calculate the norm of $\|\sum_j a_j\mathds{1}_{[x_j,y_j]}\|_{\dot{X}(\Omega_K)}$. 
\begin{Def}
Let us denote by $D_{\infty}(A)$ a cylinder of $A$, i.e. $D_{\infty}(A)=\{x\in\partial\Omega_K: A\in\operatorname{br}(x)\}$.  
We call an arc \textcolor{chn}{$[x,y]$} rational if there exists a finite sequence $A_1,...,A_k \in V_{T_K}$ s.t. $[x,y]:=\cup_{n=1}^{k} D_{\infty} (A_n)$ and we say that points x,y are rational points.
\end{Def}
For a given arc $[x,y]$ there exists a sequence of vertices $A_k\in V_{T_K}$ s.t. $[x,y]=\bigcup_k D_{\infty}(A_k)$ and sets $D_{\infty}(A_k)$ are pairwise disjoint. Moreover this sequence can be taken maximal in the sense that if \textcolor{chn}{a} vertex $A$ is in the sequence then there exists $z\in D_{\infty}(A\downarrow)$, which is not in $[x,y]$. Such \textcolor{chn}{a} sequence is unique for $[x,y]$. Let $n(k)$ be a natural number such that $A_k\in D_{n(K)}$. 
Let
\[
d([x,y])=\sum_{k} 3^{-n(k)}.
\]
We introduce an auxiliary metric on the boundary $\partial\Omega_K$
\[
d_{K}(x,y)= \min\{ d([x,y]),d([y,x])\}.
\] 
It is easy to check that $d_K(x,y)$ is a metric \textcolor{chn}{which is} greater than two dimensional euclidean metric. We prefer this metric over \textcolor{chn}{ the euclidean one} because it is a monotone function on an arc $[x,y]$ with respect to the natural order on the arc.
In the lemma below we show that for every rational arc and every monotone right continuous function on this arc there exists a "good" extension of this function to $\dot{BV}_{T_K}$.\textcolor{chn}{  We call an arc [x,y] a short arc iff
\[
d([x,y])\leq d([y,x]).
\]}
We will say that a function is monotone on an arc if it is monotone with respect to the natural order on the arc.
\begin{Le}\label{charprzyb}
Let $x,y\in \partial \Omega_K$ and $[x,y]$ be a short arc. Let function $F : \partial \Omega_K\rightarrow \mathbb{R}$ be a monotone and continuous function on the arc $[x,y]$ and $\operatorname{supp}(F)\subset [x,y]$. There exists $h\in \dot{BV}_{T_K}$ such that \begin{enumerate}
\item $\|h\|_{\dot{BV}(\Omega_K)}\lesssim (|F(x)|+|F(x)-F(y)|) d_K(x,y),$
\item $F(z)= \underset{\substack{A\in \br(z)\\ A\rightarrow z}}{\lim} \dashint h(y) dy \qquad \forall \, z\in\partial\Omega_K$.
\end{enumerate}
\end{Le}
\begin{proof} 
First we prove the existence of $\mbox{}_sh$ the good extension for characteristics functions on arcs $[s,y]\subset [x,y]$. Since arc $[x,y]$ is a short arc then an arc $[s,y]$ is a short arc and can be written as a countable sum $\bigcup_{k=1}^{M} D_{\infty}(A_k)$ in a unique way mentioned in the definition of $d_K$. From this assumption it is clear that $\#\{A_k: A_k\in D_n\} \leq 10$. Let us put
\[
_sh_A= \sum_{k} \mathds{1}_{D\uparrow\!(A_k)}(A)\textcolor{chn}{.}
\]
Clearly along every infinite branch $\operatorname{br}(z)$ the limit of $\underset{\substack{A\in \br(z)\\ A\rightarrow z}}{\lim} \mbox{}_sh_A$ exists and it is equal to $\mathds{1}_{[s,y]}(z)$. We need to estimate the total variation of $\mbox{}_sh$. From the definition of $A_k$ it follows that
\[
\|_sh\|_{\dot{BV}(\Omega_K)}=\sum_{k} 3^{-n(k)}= d_K(x,s)\leq  |\mathds{1}_{[s,y]}(y)-\mathds{1}_{[s,y]}(x)| d_K(x,y)= d_K(x,y)\textcolor{chn}{.} 
\]
Let us assume that $F$ is an increasing function. For $F$ let $\mu$ be its Lebesgue-Stieltjes measure $\mu$ i.e. $\mu ((a,b]) = F(b)- F(a)$. From the definition of the measure $\mu$ and the assumptions on $F$ we get    
\[
F(t)= F(x)+\int_{x}^{t} 1\; d\mu(s) = F(x)\mathds{1}_{[x,y]}(t) + \int_{x}^{y} \mathds{1}_{[s,y]}(t) d\mu(s)\textcolor{chn}{.}  
\]
We define $h$ by the formula
\[
h_A= F(x)_xh_A +\int_{x}^{y} \mbox{}_sh_A d\mu (s)\textcolor{chn}{.} 
\]
It follows from the definition of $h$ that
\[
\begin{split}
\|h\|_{\dot{BV}_{T_K}}= |F(x)|\|_yh\|_{\dot{BV}_{T_K}}+\int_{x}^{y} \|\mbox{}_sh\|_{\dot{BV}_{T_K}} d\mu(s)&\leq  |F(x)|d_K(x,y) + \int_{x}^{y}d_K(x,y) d\mu(s)
\\&= d_K(x,y)( |F(x)-F(y)| +|F(x)|)\textcolor{chn}{.} 
\end{split}
\]
Since $\mbox{}_sh_A\leq 1$ it follows from Lebesgue\textcolor{chn}{'s} dominated convergence theorem that
\[
\begin{split}
\underset{\substack{A\in \br(z)\\ A\rightarrow z}}{\lim} h_A&= F(x)+ \underset{\substack{A\in \br(z)\\ A\rightarrow z}}{\lim} \int_{x}^{y} \mbox{}_sh_A d\mu(s)
\\&= F(x)+\int_{x}^{y}  \underset{\substack{A\in \br(z)\\ A\rightarrow z}}{\lim} \mbox{}_sh_A d\mu(s)= F(x)+\int_{x}^{y} \mathds{1}_{[s,y]}(z) d\mu (s) = F(z) \textcolor{chn}{.} 
\end{split}
\]
\end{proof}
\begin{Le}\label{Lipschmono}
For every $\;x,y\in\partial \Omega_K$ and $a,b\in\mathbb{R}$ there is a monotone  Lipschitz function $f$, with respect to the euclidean metric, on the arc $[x,y]$ such that $f(x)=a$ and $f(y)=b$.  
\end{Le}
\begin{proof}
In order to construct such function we proceed inductively. If we have defined values in the arc $[z,t]$ only at \textcolor{chn}{a} points $z$ and $t$ we choose point $s\in [z,t]$ such that $|s-z|\geq |z-t|/2$ and $|s-t|\geq |z-t|/2$ \textcolor{chn}{(}since arc $[z,t]$ is a one dimensional curve there exists such a point\textcolor{chn}{)}. We put $f(s):= \frac{f(z)+f(t)}{2}$. \textcolor{chn}{This procedure allows us to define $f$ on a dense subset.} We extend $f$ to the whole arc. \textcolor{chn}{The} function $f$ has desired properties.  
\end{proof}
In the lemma below we prove the existence of a class of functions in $\dot{BV}$, which have desirable properties and every function from this class provides a good approximation of the norm of its trace on the boundary.
\begin{Le}\label{graniceistniejawszedzie}
Let $x,y \in \partial \Omega_K$. There are sequences of functions $f_n\in \dot{BV}(\Omega_K)$, $g_n\in C(\overline{\Omega_K})\cap \dot{BV}(\Omega_K)$, and  $h_n\in \dot{BV}(\Omega_K)$ such that:
\begin{enumerate}
\item $f_n=h_n+ g_n$,
\item For every $z\in\partial\Omega_K$,\; $\mathds{1}_{[x,y]}(z)=\underset{\substack{A\in \br(z)\\ A\rightarrow z}}{\lim} \dashint_{A} f_n(y)dy $,
\item $\|g_n\|_{\dot{BV}(\Omega_K)}\leq (1+\frac{1}{n^2})\|Tr\; g_n\|_{\dot{X}(\Omega_K)}$.
\item $Tr\; g_n$ is a Cauchy sequence in $\dot{X}(\Omega_K)$  
\item $\|h_n\|_{\dot{BV}(\Omega_K)}\leq \frac{1}{n^2}$.
\item $\|f_n\|_{\dot{BV}(\Omega_K)}\leq (1+\frac{1}{n^2})\|Tr\, g_n\|_{\dot{X}(\Omega_K)}+\frac{1}{n^2}$.
\end{enumerate} 
\end{Le}
\begin{proof}
We use Lemma \ref{Lipschmono}. For every $\varepsilon$ and every rational arc $[x,y]$, the characteristic function of $[x,y]$ can be written as sum of a Lipschitz function $g$ and a two monotone Lipschitz functions $p_1,p_2$, with supports in arcs $[t_1,x],[y,t_2]$ respectively. Moreover $t_1,t_2$ are rational,  $|t_1-x|+|t_2-y|\leq \varepsilon$ and the monotone functions $p_i$ are bounded uniformly by one. Hence from the Lemma \ref{charprzyb} for every function $p_i$ there exists a function $f^i$ such that $\|f^i\|_{\dot{BV}(\Omega_K)}\leq C \varepsilon$ and for every $z\in \partial\Omega_K$
\[
\underset{\substack{A\in \br(z)\\ A\rightarrow z}}{\lim} f^i_A=p_i(z) .
\] 
Any Lipschitz extension of $g$ to $\Omega_K$ is in $W^{1}_{1}(\Omega_K$). Hence $g$ is in the trace space. From the definition of the trace space there exists a $g_{\varepsilon}\in C(\overline{\Omega_K})\cap \dot{BV}(\Omega_K)$ such that
\[
\begin{split}
\|g_{\varepsilon}\|_{\dot{BV}(\Omega_K)}&\leq (1+\varepsilon)\|g\|_{\dot{X}(
\Omega_K)},\\
Tr g_{\varepsilon}&= g.
\end{split}
\]
Since $g_{\varepsilon}$ is in $C(\overline{\Omega_K})$ we have $ \underset{\substack{A\in \br(x)\\ A\rightarrow x}}{\lim} \dashint_{A} g_{\varepsilon}(y)dy= g(x)$. Therefore the function $f=g_{\varepsilon} + f^1+f^2= g_{\varepsilon}+h_{\varepsilon}$ has desired properties. The limits along $\operatorname{br}(z)$ of  $\dashint_{A} g_{\varepsilon}(y)dy$ exist and are equal to $\mathds{1}_{[x,y]}(z)$ for every $z\in\partial\Omega_K$ and
\[
\|f\|_{\dot{BV}(\Omega_K)}\leq (1+\varepsilon)\|g\|_{\dot{X}(\Omega_K)} + C\varepsilon,
\] 
where the term $C\varepsilon$ is the estimate on the norms of the functions $f^i$. For every $n$ we choose suitable $\varepsilon$ and we get \textcolor{chn}{the} desired properties. The sequence $Tr g_n$ is Cauchy sequence. Indeed for a given function $g_n$ and $m>n$ there exists a continuous piecewise monotone function $q$ with support on a small set on the boundary such that
\[
q + Tr g_n=  Tr g_m  ,
\]
From Lemma \ref{charprzyb} there exists a function $\tilde{q}\in\dot{BV}(\Omega_K)$ with a small norm such that
\[
Tr (g_n+ \tilde{q})= Tr g_m \textcolor{chn}{.} 
\]
The size of the support of $q$ depends only on $g_n$. Therefore
\[
\|Tr g_n - Tr g_m\|_{\dot{X}(\Omega_K)}\leq \epsilon
\]
for sufficiently large $n,m$. 
\end{proof}
The Cauchy sequence \textcolor{chn}{$ \{ Tr\, g_n\}$ } defines an element \textcolor{chn}{$u\in \dot{X}(\Omega_K)$}. From the analogous argument as in the above Lemma if $f\in \dot{BV}(\Omega_K)$ satisfies $\mathds{1}_{[x,y]}(z)=\underset{\substack{A\in \br(z)\\ A\rightarrow z}}{\lim} \dashint_{A} f(y)dy $ for every $z$ on the boundary then \textcolor{chn}{$Tr f =u$}. To simplify \textcolor{chn}{the} notation we denote \textcolor{chn}{$u= \mathds{1}_{[x,y]}$}. From the point 6. of the Lemma \ref{graniceistniejawszedzie} it follows
\[
\|\textcolor{chn}{u}\|_{\dot{X}(\Omega_K)}=\lim_{n\rightarrow\infty}  \|Tr\, g_n\|_{\dot{X}(\Omega_K)}= \lim_{n\rightarrow\infty} \|f_n\|_{\dot{BV}(\Omega_K)}
\]
Since the projection from $\dot{BV}$ onto $\dot{BV}_{T_K}$ preserves the trace, we may assume that functions $f_n$ are from $\dot{BV}_{T_K}$. Therefore the function $g=\sum_{j} a_j \mathds{1}[x_j,y_k]$, whose arcs $[x_j,y_j]$ are rational, satisfies
\[
\|g\|_{\dot{X}(\Omega_K)}\simeq\inf\{\|f\|_{\dot{BV}_{T_K}} : f\in L\mbox{ and } Trf=g\},
\] 
where $L\subset\dot{BV}_{T_K}$ consists of such $f$ that the limit $ \underset{\substack{A\in \br(x)\\ A\rightarrow x}}{\lim} f_A $ exists for every $x\in \partial \Omega_K$ and it is equal to $Tr f(x)$.  

\begin{Rem} In the above lemmas we abuse the notation a bit. For rational points $x$ there are two branches $\operatorname{br}(x)$. If we look at a finite linear combination of characteristic functions of arcs, the are finitely many points (endpoints of segments) on which the limits over this two the branches are different. However they are equal to the value of the trace either on left or right side of that endpoint. Further in the article we are only interested in branches which contain some specific vertex $A$. Hence we are interested only in one of the problematic branches and it is clear what we mean by the limit.
\end{Rem} 
We want to characterize the space $\dot{X}(\Omega_K)$. We introduce, a metric on von Koch's curve by \textcolor{chn}{the} formula
\[
\tilde{d}(x,y):= \|\mathds{1}_{[x,y]}\|_{\dot{X}(\Omega_K)},
\]
where $\mathds{1}_{[x,y]}$ is a characteristic function of an arc on the von Koch's curve which connects $x$ and $y$. It does not matter which one of the two arcs we take because the difference between their characteristic functions is  constant. Further in the proof it will be clear which arc is considered.
Since $\|\cdot\|_{\dot{X}(\Omega_K)}$ is a norm, $\tilde{d}$ is a metric on the boundary. For a given metric space $(Y,d_Y)$ we define the Arens-Eells space (\cite{Weaver1999}).
\begin{Def} 
Let $(Y,d_Y)$ be a metric space. We call a function $f:Y\rightarrow \Rr$ a molecule if it has finite support and $\sum_{y\in Y} f(y)=0$. Let $x,y\in Y$. We define special type of a molecule - an atom : $m_{xy}=\mathds{1}_{x}-\mathds{1}_{y}$, where $\mathds{1}_a$ is a characteristic of a set $\textcolor{chn}{\{a\}}$. Let $m$ be a molecule, i.e. $m= \sum_{j=1}^M a_j m_{x_j y_j}$, then the Arens-Eells norm of $m$ is 
\[
\|m\|_{AE(d_Y)}=\inf\left\{\sum_{j} |a_j|d_Y(x_j,y_j): m:= \sum_{j} a_j m_{x_jy_j}\right\},
\]
where the infimum is taken over all possible representations of $m$ as a sum of $m_{pq}$. The Aerens-Eells space is the completion of molecules with respect to the norm $\|\cdot\|_{AE}$.
\end{Def}
We want to show that $\dot{X}(\Omega_K)$ is isomorphic to the Arens-Eells space with the metric $\tilde{d}$. We will denote by $M(\tilde{d})$ the linear space of molecules. Clearly it is a non-complete norm space. By the definition it is dense in $AE(\tilde{d})$. We define the candidate for the isomorphism on the a linearly dense subsets of both spaces. We set $\Psi: AE(\tilde{d})\rightarrow\dot{X}(\Omega_K)$ by the formula
\begin{equation}\label{definicjaPSI}
\Psi(m_{xy})= \mathds{1}_{[x,y]}\qquad \forall \; x,y\in\partial \Omega_K.
\end{equation}
\begin{Le}
$\Psi :  AE(\tilde{d})\rightarrow\dot{X}(\Omega_K) $ is an isomorphism between Banach spaces.
\end{Le}
\begin{proof}
By \textcolor{chn}{the} triangle inequality and the definitions of $\tilde{d}(x,y)$ and Arens-Eells space, it follows that $\Psi$ is continuous
\begin{equation}\label{kontrakcjapsi}
\|\Psi(f)\|_{\dot{X}(\Omega_K)}\leq \|f\|_{AE(\tilde{d})}.
\end{equation}
Proving the estimate from below is more involved. In the trace space we have following density result.
\begin{Le}
$\Phi(M(\tilde{d}))$ is dense in $\dot{X}(\Omega_K)$.
\end{Le}
\begin{proof}
From \cite{MR3519964} we know that \textcolor{chn}{the} restrictions of Lipschitz functions on $\Rr^2$ are dense in $W^{1}_{1} (\Omega_K)$. Therefore Lipschitz functions are dense in $\dot{X}(\Omega_K)$. Hence for any $f\in\dot{X}(\Omega_K)$ there exists a sequence of Lipschitz functions $f_n$  such that
\[
\lim_{n\rightarrow \infty}\|f-f_n\|_{\dot{X}(\Omega_K)}=0.
\]
So it is enough to approximate Lipschitz functions with piecewise constant functions. Let $f$ be a Lipschitz function. We define \textcolor{chn}{a} piecewise constant function $g_k= \sum \min\{f(x):x\in [x_j,x_{j+1}]\}\mathds{1}_{[x_j,x_{j+1}]}$, where $x_j$ are rational points of order $k$ i.e. $\exists A\in D_k$ such that $[x_j,x_{j+1}]=D_{\infty}(A)$. We define \textcolor{chn}{a} function \textcolor{chn}{$h$ by the formula}
\[
h_A=\inf\{f(z) - g_k(z): z \in D_{\infty}(A)\}\textcolor{chn}{.}
\]
\textcolor{chn}{The} function $h$ satisfies
\[
Tr \,h = f-g_k.
\]
We will estimate the $\dot{BV}_{T_K}$ norm of \textcolor{chn}{the function} $h$.
The function $f$ is also Lipschitz with respect to the metric $d_K$. Let $K$ be \textcolor{chn}{the} Lipschitz constant of $f$ with respect to $d_K$. 
Observe that due to \textcolor{chn}{the} Lipschitz continuity of the function $f$ there are positive numbers $\{b_i\}_{i=1}^{5}, \{c_i\}_{i=1}^{3}$ such that for every pants shaped polygon $A\in D_n$ and $n\geq k$ we have
\[
\frac{1}{3^n}\sum_{Q\in D_1(A)} |h(A)-h(Q)| = \frac{1}{3^n}\sum_{i=1}^{5}\frac{ b_i}{3^n}\leq K \max_{i} b_i \frac{\# D_1(A)}{9^n} .
\]
Similarly for palace shaped polygon $B$
\[
\frac{1}{3^n}\sum_{Q\in D_1(B)} |h(B)-h(Q)| = \frac{1}{3^n}\sum_{i=1}^{3}\frac{c_i}{3^n}\leq K\max_{i} c_i \frac{\# D_1(B)}{9^n} .
\] 
Let $\rho:=\max\{b_1,\ldots,\, b_5,\,c_1,\, c_2,\,c_3\}$. We can prove inductively that $\# D_j(A)\lesssim 4^{j}$. Let $A\in D_k$ we have \textcolor{chn}{the} following estimate on the variation on the sub-tree $D\!\!\uparrow\!(A)$, starting with $A\downarrow$
\[
\begin{split}
\frac{1}{3^{k-1}}|h_{A_k}-0|+\sum_{i=1}^{\infty} \sum_{Q\in D_i(A)} |h_B-h_{B\downarrow}|\frac{1}{3^{k-1+i}}&\leq K\rho \sum_{j=k}^{\infty}\frac{\# D_{j-k}(A)}{9^j}
\\&\lesssim K\sum_{j=k}^{\infty}\frac{4^{j-k}}{9^j}\lesssim  \frac{K}{9^{k}}\textcolor{chn}{.}
\end{split}
\]
We sum \textcolor{chn}{the} above inequalities over all $A\in D_k$ and we get
\[
\|h\|_{\dot{BV}_{T_K}}\lesssim K\frac{4^k}{3^{2k}}\textcolor{chn}{.}
\]
\textcolor{chn}{The} left hand side tends to zero with $k\rightarrow \infty$. Hence $\Psi(M(\tilde{d}))$ is dense in $\dot{X}(\Omega_K)$. 
\end{proof}

To show that $\Psi$ is an isomorphism we need to prove the estimate from below on the norm of $\Psi(m)$. The next auxiliary lemma reduces our problem to a finite tree.
\begin{Le}\label{ucieciewnieskon}
Let $f\in L$ and $Tr f(z)=c$ for every $z\in [x,y]$. Function $\tilde{f}\in L$ given by the formula
\[
\tilde{f}_A=\left\{\begin{array}{cl} c &\qquad   D_{\infty}(A)\subset[x,y],\\
                        f_A&\qquad  \mbox{ in a opposite case},\end{array}\right.
\]
satisfies
\[
\|\tilde{f}\|_{\dot{BV}_{T_K}}\leq \|f\|_{\dot{BV}_{T_K}}.
\]
\end{Le}
\begin{proof}
Fix $A_0\in V_T$ such that $D_{\infty}(A_0)\subset[x,y]$. Without loss of generality we assume that $f_{A_0}=0$ and $c=1$. If $B$ is a descendant of $A_0$ it follows from the definition that $D_{\infty}(B)\subset[x,y]$. 

We can assume that for $B\in D\!\uparrow(A_0)$ the value $f_B$ does not exceed one. Indeed if $B$ is such that $f_{B\downarrow}\leq 1$ and $f_B>1$ then we define an auxiliary function $h$
\[
h_Q=\left\{\begin{array}{cl} 1 &\qquad  Q=B\mbox{ or } Q\in D\!\uparrow(B),\\
                        f_Q&\qquad  \mbox{ in a opposite case},\end{array}\right.
\]
\textcolor{chn}{The} function $h$ has the same trace as $f$ and differs from $f$ only on $D\!\uparrow(B)$. Since
\[
|f_{B\downarrow}-f_B|>|f_{B\downarrow}-1|
\]
and $h$ is constant on $D\!\uparrow(B)$ it follows that
\[
\|h\|_{\dot{BV}_{T_K}}< \|f\|_{\dot{BV}_{T_K}}.
\]
We can assume that $f$ is monotone (non-decreasing) on $D\!\!\uparrow\!\!(A_0)$ with respect to \textcolor{chn}{the} descendancy relation i.e. if $B\in D\!\!\uparrow\!\!(A_0)$ and $C$ is a descendant of $B$ then $f_B\leq f_C$. Indeed suppose that $f_C<f_B<1$ for some $C\in D_1(B)$. Since for functions in $L$ the value of trace $Tr f(x)$ is defined as \textcolor{chn}{the} limit along $\operatorname{br}(x)$, but for $x\in D_{\infty}(A)$ the limit is one. Therefore on every branch $ \operatorname{br}(x)$ such that $x\in D_{\infty}(C)$ there exists a vertex $Q$ such that $f_Q\geq f_B$ and $f_{Q\downarrow} < f_B$. We denote by $\omega(C)$ the set of all such vertices. Let $T(C)$ be a tree with a root $C$ and \textcolor{chn}{the} set of leafs is equal to $\{Q\downarrow: Q\in \omega(C)\}$. We define \textcolor{chn}{an} auxiliary function $p$ by the formula
\[
p_Q=\left\{\begin{array}{cl} f_B &\qquad  Q\in V_{T(C)},\\
                        f_Q&\qquad  \mbox{ in a opposite case},\end{array}\right.
\] 
On the tree $T(C)$ the variation of  $p$ is equal to the weighted sum of differences on leafs. However for every $Q\in\omega(C)$
\[
|p_Q-p_{Q\downarrow}|=|f_Q-f_B|\geq |f_Q-f_{Q\downarrow}|.
\] 
Therefore
\[
\|p\|_{\dot{BV}_{T_K}}< \|f\|_{\dot{BV}_{T_K}}.
\]
We have reduced our problem to the set of functions $Y(f)\subset L$ such that $h\in Y(f)$ iff it is a non-decreasing function on $D\!\!\uparrow\!\!(A_0)$ with respect to \textcolor{chn}{the} descendancy relation, $h_B=f_B$ for every $B\in V_{T_K}\backslash D\!\!\uparrow\!\!(A_0)$ and $Tr\, h(x)=1$ for $x\in D_{\infty}(A_0)$. We introduce a partial order on $Y(f)$. For $h,\; z\in Y(f)$  
\[
h\preceq z\quad\Leftrightarrow\quad \forall A\in V_{T_K}\quad h_A\leq z_A\quad\mbox{ and }\quad\|z\|_{\dot{BV}_{T_K}}\leq \|h\|_{\dot{BV}_{T_K}}.
\] 
If $C\subset Y(f)$ is a chain with respect to the relation $\preceq$ then it has an upper bound in $Y(f)$. Indeed the function $z\in Y(f)$ defined by the formula
\[
z_A= \sup_{u\in C} u_A
\]
is an upper bound. Function $z$ is a supremum of non-decreasing functions hence it is non-decreasing. If every non-decreasing sequence  $b_{\alpha}^k$ is convergent to one as $k\to \infty$ then $\sup_{\alpha} b_{\alpha}^k$ converges to one. Therefore $z$ has the same trace as functions in $Y(f)$. In particular $Tr\, h:=1$ for $x\in D_{\infty}(A_0)$. 
By the definition if $ u\preceq v$ then $u_Q\leq v_Q$ for every $Q\in V_{T_K}$ and the total variation \mbox{$\|v\|_{\dot{BV}_{T_K}}\leq \|u\|_{\dot{BV}_{T_K}}$}. Hence for every $n$ we can choose a sequence $f^k \in Y(f)$ such that 
\[
\lim_{k\rightarrow \infty}\|f^k\|_{\dot{BV}_{T_K}}=\inf_{u\in C}\|u\|_{\dot{BV}_{T_K}}.
\]
and $\lim f^k_Q = z_Q$ for every $Q\in \bigcup_{j=1}^{n} D_j $. Therefore \textcolor{chn}{the} following estimate is satisfied
\[
\sum_{j=1}^{n}\sum_{Q\in D_j} \frac{1}{3^j} |z_Q-z_{Q\downarrow}|\leq \inf_{u\in C}\|u\|_{\dot{BV}_{T_K}}
\]
Taking \textcolor{chn}{the} limit with $n\rightarrow\infty$ we get
\[
\|z\|_{\dot{BV}_{T_K}}< \inf_{u\in C}\|u\|_{\dot{BV}_{T_K}}.
\]
Since every chain in $Y(f)$ has an upper bound in $Y(f)$ by the Kuratowski-Zorn Lemma, there exists \textcolor{chn}{an} element of $Y(f)$ maximal with respect to $\preceq$. Let $w\in Y(f)$ be \textcolor{chn}{the} maximal element. By the monotonicity of $w$, it follows that $w_{Q\downarrow}\leq w_{Q}$ for every $Q\in D\!\!\uparrow\!\!(A_0)$. Since for every $Q\in V_{T_K}$ the set of direct descendants $D_1(Q)$ has at least three elements,
\[
\begin{split}
|w_{Q\downarrow}-w_{Q}|+ \sum_{B\in D_1(Q)} \frac{1}{3} | w_B-w_Q|&=w_{Q}-w_{Q\downarrow} + \sum_{B\in D_1(Q)} \frac{1}{3}  w_B-w_Q 
\\&=(1-\frac{\# D_1(Q)}{3})w_Q-w_{Q\downarrow}   +\sum_{B\in D_1(Q)} \frac{1}{3}  w_B
\\&\geq (1-\frac{\# D_1(Q)}{3})\min_{B\in D_1(Q)}(w_B)
\\&\quad-w_{Q\downarrow}   +\sum_{B\in D_1(Q)} \frac{1}{3}w_B.
\end{split}  
\]
\textcolor{chn}{The} function $w$ is maximal with respect to $\preceq$, hence $w_Q=\min_{B\in D_1(Q)} w_B$ for every $Q\in D\!\!\uparrow\!\!(A_0)$. Therefore there is an infinite branch $\operatorname{br}(x)$ such that $x\in D_{\infty}(A_0)$ and $w$ is constant on $\operatorname{br}(x)\cap D\!\!\uparrow\!\!(A_0)$. However for $x\in D_{\infty}(A_0)$ the limit over any branch $\operatorname{br}(x)$ is equal to one. Hence $h_B=1$ for every $B\in D\!\!\uparrow\!\!(A_0)$. We have proven that changing the values of $f$ to one on the descendants of $A_0$ does not increase the total variation. It remains to consider the value at the point $A_0$. By the triangle inequality and the fact that for every vertex $Q$, $\# D_1(Q)\geq 3$ we have
\[
|f_{A_0\downarrow}-f_{A_0}|+\sum_{B\in D_1(A)} \frac{1}{3} | 1-f_{A_0\downarrow}|\geq |f_{A_0\downarrow} -1|.
\]   
Therefore changing the value of $f$ on $A_0$ and its descendants to one, will not increase the total variation. Since \textcolor{chn}{the} only assumption on $A_0$ was that $D_{\infty}(A_0)\subset [x,y]$ we have desired estimate
\[
\|\tilde{f}\|_{\dot{BV}_{T_K}}\leq \|f\|_{\dot{BV}_{T_K}}.
\]
\end{proof}

\begin{Le}
Let $A_0\in D_n$ and $[x,y]=D_{\infty}(A_0)$ then 
\begin{equation}\label{chichi}
\tilde{d}(x,y)= 3^{-n}.
\end{equation}
\end{Le}
\begin{proof}
For any $f\in\dot{BV}_{T_K}$ such that $Trf= \mathds{1}_{[x,y]}$ we define
\[
\tilde{f}_A=\left\{\begin{array}{cl} 1 &\qquad   D_{\infty}(A)\subset[x,y],\\
                        f_A&\qquad  A\in D_{k}, k\leq n,\\
                        0&\mbox{ in a opposite case.}\end{array}\right.
\] 
From the Lemma \ref{ucieciewnieskon} it follows that 
\[
\|\tilde{f}\|_{\dot{BV}_{T_K}}\leq \|f\|_{\dot{BV}_{T_K}}\textcolor{chn}{.}
\]
However
\[
\|\tilde{f}\|_{\dot{BV}_{T_K}}\geq \frac{1}{3^{n}} \sum_{B \in D_1(A_0\downarrow)} |f_{A_0\downarrow}-f_B|\geq \frac{1}{3^{n}}\left( |f_{A_0\downarrow} - 1 |+|f_{A_0}|\right)\geq \frac{1}{3^n}.
\]
The right hand side of the inequality is \textcolor{chn}{the} total variation of a function $p$ \textcolor{chn}{given by the formula} 
\[
p_A=\left\{\begin{array}{cl} 1 &\qquad   D_{\infty}(A)\subset[x,y],\\
                                      0&\mbox{ in a opposite case.}\end{array}\right. 
\]
\end{proof}
Let us observe that the set of functions $\sum_j a_j\mathds{1}_{[x_j,y_j]}$, where $x_j,y_j$ are rational, is dense in $\dot{X}(\Omega_K)$. Indeed for every irrational arc $[x,y]$ there exists a sequence of points $t_n,z_n$ such that
\[
\|\mathds{1}_{[x,y]} - \mathds{1}_{[t_n,z_n]}\|_{\dot{X}(\Omega_K)}\lesssim \frac{1}{3^n}.
\]
Similarly we observe that  molecules $\sum_{j} a_j m_{x_j y_j}$, where $x_j,y_j$ are rational, are dense in Arens-Eells space. \\
We fix $g=\sum_j a_j\mathds{1}_{[x_j,y_j]}$, where arcs $[x_j,y_j]$ are rational and pairwise disjoint. Let \mbox{$f\in L$} be any function such that $Tr f= g$. There exists $n_0=n_0(g)$ such that for $A\in D_{n_0}$ either there exists an arc $[x_j,y_j]$ such that $D_{\infty}(A)\subset [x_j,y_j]$ or $D_{\infty}(A)$ and $\bigcup [x_k,y_k]$ are disjoint. 
We define \textcolor{chn}{the} function $Wf\in L$ by
\[
Wf_A=\left\{\begin{array}{cl} a_j &\qquad D_{\infty}(A)\subset [x_j,y_j]\\
0 &\qquad D_{\infty}(A)\cap \bigcup_{j} [x_j,y_j]=\emptyset,\\
f_A &\qquad\mbox{in other cases.}\end{array}\right.
\]
It is easy to observe that $Trf =TrWf$. Moreover from Lemma \ref{ucieciewnieskon} it follows that
\[
\|Wf\|_{\dot{BV}_{T_K}}\leq \|f\|_{\dot{BV}_{T_K}}. 
\]
Therefore
\[
\inf\{\|f\|_{\dot{BV}_{T_K}}:\; Tr f=g\}=\inf\{\|f\|_{\dot{BV}_{T_K}}:\; Tr f=g \mbox{ and } f=Wf\}.
\]
Since we minimize the total variation over the set $\{ Tr f=g \mbox{ and } f=Wf\}$, the values $f_A$ are fixed for $A\in D_k$, $k>n_0$. Therefore the total variation on this set is a function of finitely many variables. Moreover it is a piecewise linear function with finitely many pieces. Therefore the minimum is attained. We denote the total variation minimizer by $\psi$. We define by $\gamma^A\in\dot{BV}_{T_K}$ 
\[
\gamma^A_B=\left\{\begin{array}{cr} 1& B\in D\uparrow(A),
\\ 0&\quad \mbox{ in other cases}.\end{array}\right.
\]
Therefore \textcolor{chn}{by Abel's} summation formula
\[
\psi= \psi_R +\sum_{j=1}^{n_0} \sum_{A\in D_n} \left(\psi_{A} -\psi_{A\downarrow}\right) \gamma^A. 
\]
A simple calculation gives us 
\begin{equation}\label{kartaamberu}
\|\psi\|_{\dot{BV}_{T_K}}= \sum_{j=1}^{n_0} \sum_{A\in D_n} |\psi_{A} -\psi_{A\downarrow}|\;\| \gamma^A\|_{\dot{BV}_{T_K}}.
\end{equation}
 The function $\|\psi\|_{\dot{BV}_{T_K}}$ minimize the variation for a given trace, hence 
\[
\|Tr{f}\|_{\dot{X}(\Omega_K)}=\|\psi\|_{\dot{BV}_{T_K}}.
\]
Therefore from $\eqref{kartaamberu}$, $\eqref{chichi}$
\[
\begin{split}
\|Tr\,\psi\|_{\dot{X}(\Omega_K)}&\simeq \sum_{j=1}^{n_0} \sum_{A\in D_n} |\psi_{A} -\psi_{A\downarrow}|\;d(x(A),y(A))
\\&\geq \|\sum_{j} a_j m_{x_j y_j}\|_{AE(\tilde{d})}\textcolor{chn}{.}
\end{split}
\]

Therefore $\Psi$ is an isomorphism of Banach spaces.
\
\end{proof}
We have proven that the trace space is isomorphic to the Arens-Eells space.

 We will characterize $AE(\tilde{d})$ further.
\begin{Le}
$AE(\tilde{d})$ is isomorphic to $\ell^1$
\end{Le}  
\begin{proof}
In order to characterize $AE(\tilde{d})$ we introduce another metric on the von Koch's curve. The von Koch's curve is constructed inductively. The induction starts with a triangle and every segment of the triangle is replaced with a piecewise linear curve $w$. This curve is made of from 4 segments. In the next step every old segment is replaced with a rescaled copy of $w$. Every segment is indexed in the following way. The segment $S_x$ is replaced with segments $S_{x,0},S_{x,1},S_{x,2},S_{x,3}$.\\
\begin{center}
\begin{tikzpicture}
\draw [black]  (0,0) -- (5,0) node[draw=none,fill=none,font=\scriptsize,midway,below] {x};
\draw [black]  (7,0) -- (9,0)node[draw=none,fill=none,font=\scriptsize,midway,below] {x,1} -- (10,1.71)node[draw=none,fill=none,font=\scriptsize,midway,left] {x,2} -- (11,0)node[draw=none,fill=none,font=\scriptsize,midway,right] {x,3}--(13,0)node[draw=none,fill=none,font=\scriptsize,midway,below] {x,4};
\draw[-stealth] (5.5,0)--(6.5,0);
\end{tikzpicture}
\end{center}
$I=\{x=(x_1,x_2\ldots): x_1\in\{0,1,2\}, x_i\in\{0,1,2,3\} \mbox{ for } i>1\}$
is a set of all infinite indices of segments in the von Koch's curve construction. For every point $x\in\partial\Omega_K$ there is a corresponding index $i(x)\in I$ such that segments $S_{i(x)_1,...,i(x)_k}\to x$ as $k\rightarrow\infty$. We define a bijection between set of indices and a one dimensional Torus with the euclidean metric
\[
\Tt=\{y: y=\frac{i(x)_1}{3}+\sum_{j=2}^{\infty}\frac{i(x)_j}{4^j}\,\quad i(x)\in I\}.
\]
Every $x\in\partial\Omega_K$ has a unique index in $\Tt$. Abusing notation we denote it by $i(x)$. We can define a metric on $\partial\Omega_K$ by
\[
d(x,y):=d_{\Tt}(i(x),i(y)).
\]
As is easily on a Figure \ref{KOchsnow}, if $A\in D_n$ is a "pants" shaped polygon then $D_{\infty}(A)=[x,y]$, where $d(i(y),i(x))=\frac{2}{4^n}$. It is so because its descendants cover two segments of $n$-th generation. Similarly if $A$ is a "palace" shaped polygon, $d(i(y),i(x))=\frac{1}{4^n}$. In any of the above cases we have
\[
\tilde{d}(x,y)\simeq \frac{1}{3^n}=\frac{1}{4^{n\log_4(3)}}\simeq d(x,y)^{log_4(3)}.
\] 
For rational points $x,y$ we define
\[
f^{[x,y]}_A:=\left\{\begin{array}{cr} 1&\quad D_{\infty}(A)\subset [x,y],\,\\
							0&\quad\mbox{ otherwise.	} \end{array}\right.
\] 
Obviously $Trf^{[x,y]}:=\mathds{1}_{[x,y]}$. Since $x,y$ are rational, there exists unique finite sequence of $\{A_k\}_{k\in I}\subset V_T$, such that $f^{[x,y]}= \sum_k \gamma^{A_k}$. Let $m=\min\{n: \exists\; k\; A_k\in D_n\}$. From the definition of $f^{[x,y]}$ we deduce that $\gamma^{A_k}$ have disjoint \textcolor{chn}{supports}, and for every $n$ there are at most $10$ polygons in $\{A_k\}_{k\in I}\cap D_n$. Therefore
\[
d(x,y)= \sum_{k} d(x(A_k),y(A_k))\leq 10 \sum_{i=m} \frac{1}{4^i}\simeq \frac{1}{4^m}.
\] 
 and we have \textcolor{chn}{an} analogous estimate for $\tilde{d}$. Hence 
\[
\tilde{d}(x,y)\simeq\frac{1}{3^m}=\frac{1}{4^{m\log_4(3)}}\simeq d(x,y)^{\log_4(3)}.
\] 
Therefore $AE(\tilde{d})\cong AE(d^{\log_4(3)})$. Since $0<\log_4(3)<1$ the claim of the lemma follows from the theorem below,
\begin{Th}
Let $N\in\Nn$ and $X$ is isometric to \textcolor{chn}{an} infinite compact subset of $\Rr^N$. If $d$, $\tilde{d}$ are metrics on $X$ s.t $\tilde{d}\simeq d^{\alpha}$ for $0<\alpha<1$ then the space $AE(\tilde{d})$ is isomorphic to $\ell^1$. 
\end{Th} 
The case N=1 was proven by Z. Ciesielski \cite{MR0132389} and for $N>1$ the above Theorem follows from Theorem 3.5.5 and Theorem 3.3.3 in \cite{Weaver1999}.
\end{proof}
Therefore $\dot{X}(\Omega_K)$ is isomorphic to $\ell^1$. Let $\dot{X}(\Omega_K)=\operatorname{span}\{e_i\}$. From the definition of the trace space for every $e_i$ there exists $f_i\in \dot{BV}_{T_K}$ such that $\|f_i\|_{\dot{BV}_{T_K}} \leq 2\|e_i\|_{\dot{X}(\Omega_K)}$ and $Tr f_i=e_i$. Hence the $S$ given by the formula
\[
S\left(\sum_{i} a_i e_i\right)= \sum_{i} a_i f_i
\]
is the desired right inverse operator with $\|S\|\leq 2$. Indeed 
\[
Tr\left(S\left(\sum_i a_i e_i\right)\right)= Tr \left(\sum_{i} a_i f_i\right)=\sum_{i} a_i e_i.  
\] 
This concludes the proof of Theorem \ref{twierdzenieKocha}.
\newpage
\section*{Appendix}
In this section we show the steps of the construction of the Whitney Covering of the von Koch's snowflake used in our proof. We divide the von Koch's snowflake into six identical parts. We focus our attention on one of them.
 \begin{figure}[h]
\begin{center}
\includegraphics[width=10em]{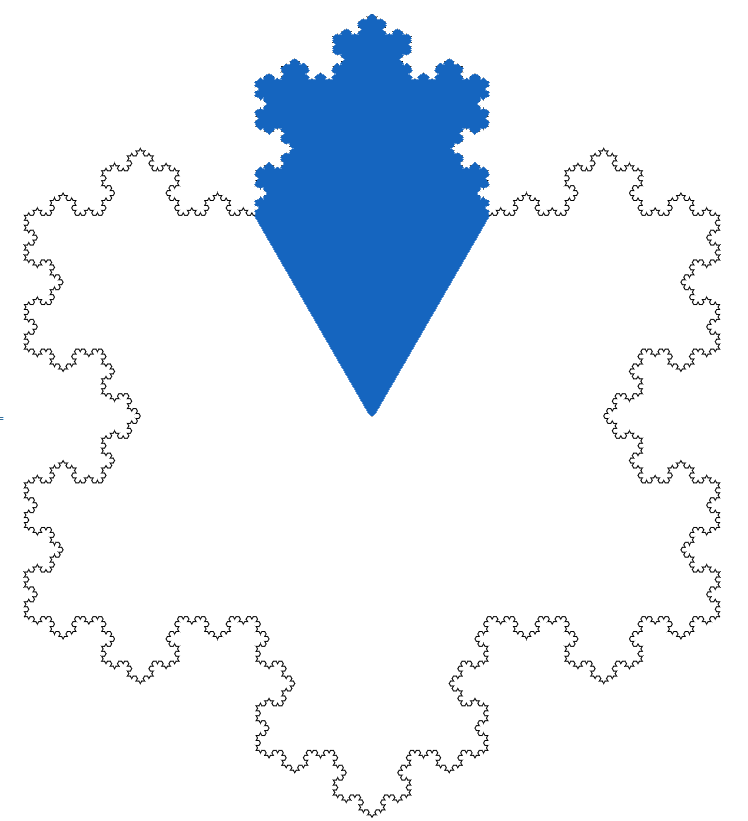}
\caption{One of the six parts of the von Koch's snowflake.}\label{KON0}
\end{center}
\end{figure}
We start our construction with the third step of the iterative construction of von Koch's snowflake. In most steps of the construction we limit the description of it to showing pictures. In our construction all the lines we use are parallel to the sides of the equilateral triangle (\textcolor{chn}{the} starting point for von Koch's snowflake construction).  We denote the vectors pointing in those three directions respectively by $\nu_1,\nu_2,\nu_3$ and by $L(X,v)$ we denote \textcolor{chn}{the} line parrallel to a vector $v$, which contains point $X$. In the first step we construct vertices of the first generation of polygons - "pants" polygon.
 \begin{figure}[h]
\begin{center}
\includegraphics[width=14em]{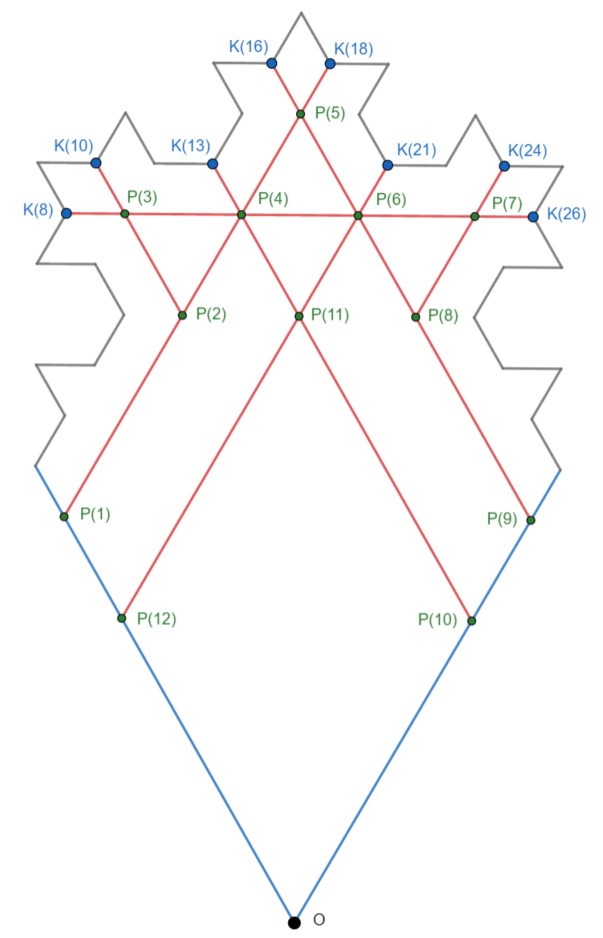}
\includegraphics[width=14em]{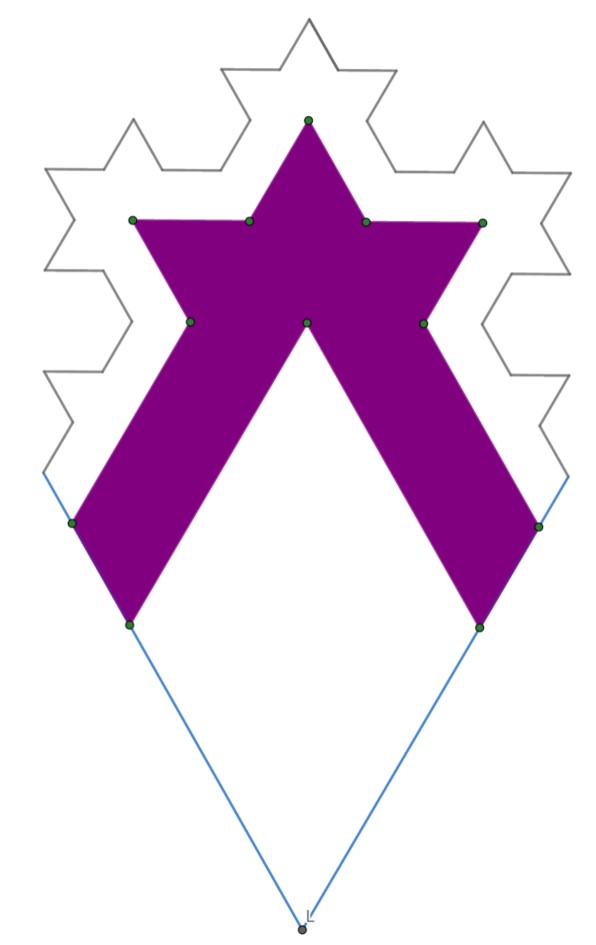}
\caption{First step of the inductive construction of the Whitney Covering.}\label{KON1}
\end{center}
\end{figure}\newpage
The green points are intersections of following lines:
\[
\begin{array}{cc} P(1)=L(O,\nu_1)\cap L(K(18),\nu_3);\quad& P(2)=L(K(10),\nu_1)\cap L(K(18),\nu_3);\\P(3)=L(K(10),\nu_1)\cap L(K(8),\nu_2);\quad&P(4)=L(K(13),\nu_1)\cap L(K(8),\nu_2);\\P(5)=L(K(16),\nu_1)\cap L(K(18),\nu_3);\quad&P(6)=L(K(16),\nu_1)\cap L(K(21),\nu_3);\\
 P(7)=L(K(26),\nu_2)\cap L(K(24),\nu_3);\quad& P(8)=L(K(16),\nu_1)\cap L(K(24),\nu_3);\\P(9)=L(K(16),\nu_1)\cap L(O,\nu_3);\quad&P(10)=L(K(13),\nu_1)\cap L(O,\nu_3);\\P(11)=L(K(13),\nu_1)\cap L(K(21),\nu_3);\quad&P(12)=L(O,\nu_1)\cap L(K(21),\nu_3);\\
\end{array}
\]
Now we take next generation of the approximation of von Koch's snowflake. We observe that we can cover the boundary using few blue and lime regions.
 \begin{figure}[h]
\begin{center}
\includegraphics[width=13em]{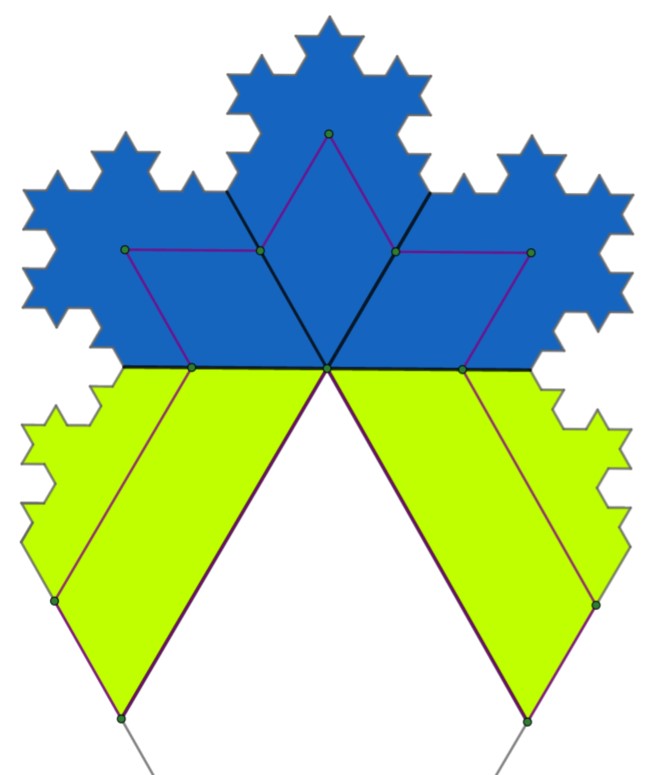}
\caption{Five regions of the second step of the construction.}\label{KON3}
\end{center}
\end{figure} \\ 
In blue regions we repeat the first step of the construction. In the lime part we do the following (Figure \ref{KON4}).
\[
\begin{array}{cc} R(1)=L(K(10),\nu_1)\cap L(K(8),\nu_2);\quad& R(2)=L(K(13),\nu_1)\cap L(K(8),\nu_2);\\R(3)=L(K(10),\nu_1)\cap L(R(2),\nu_3);\quad&R(4)=L(K(17),\nu_2)\cap L(R(2),\nu_3);\\R(5)=L(K(1),\nu_1)\cap L(R(2),\nu_3);
 \end{array}
 \]
 and $R(6)=P(2)$, $R(7)=P(1)$, where $P(1),P(2)$ are points from the previous step of construction. In the end we have constructed five new polygons. What is important vertices on the boundary of given region coincide with the corresponding vertices from the neighboring regions (Figure \ref{KON6}).\newpage
 \begin{figure}[h]
\begin{center}
\includegraphics[width=15.5em]{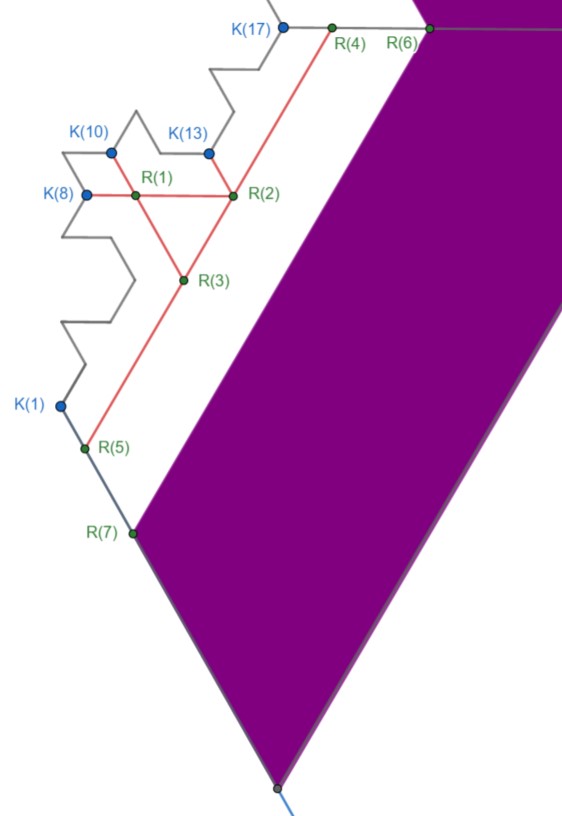}
\includegraphics[width=14.5em]{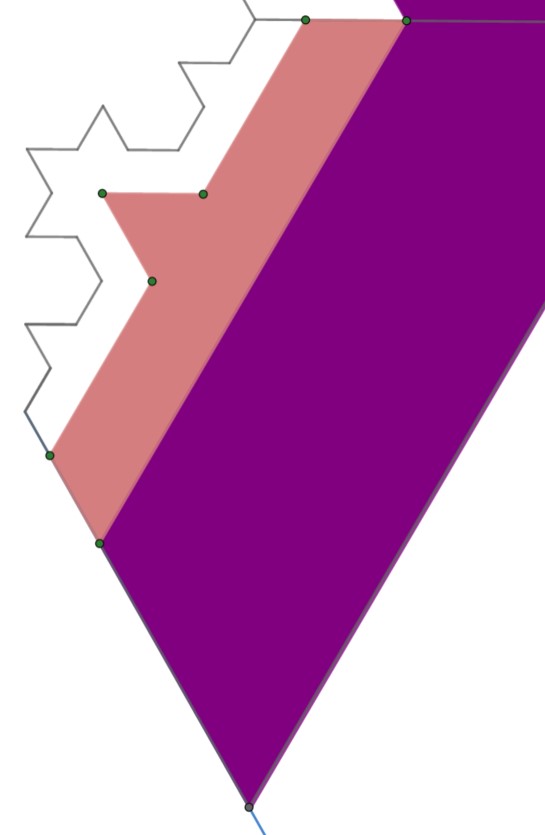}
\caption{Construction step for the yellow region.}\label{KON4}
\end{center}
\end{figure}
\begin{figure}[h]
\begin{center}
\includegraphics[width=11.5em]{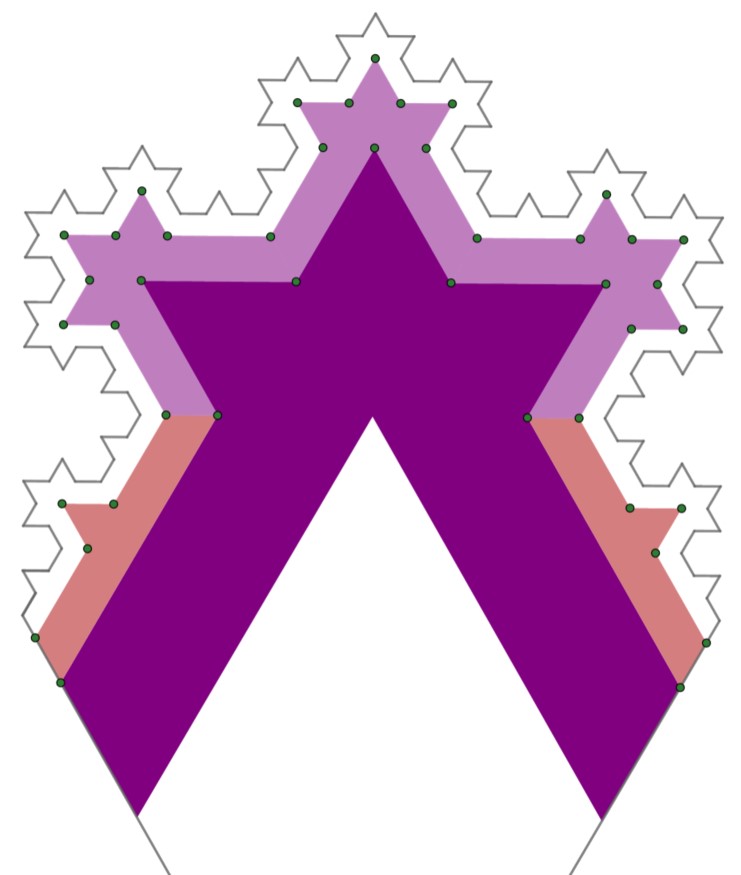}
\caption{Second generation of the polygons.}\label{KON6}
\end{center}
\end{figure}
\newpage
We take the next iterative step of the approximation of von Koch's snowflake and we observe that again we can cover the neighborhood of the boundary by the lime and blue regions.
\begin{figure}[h]
\begin{center}
\includegraphics[width=14em]{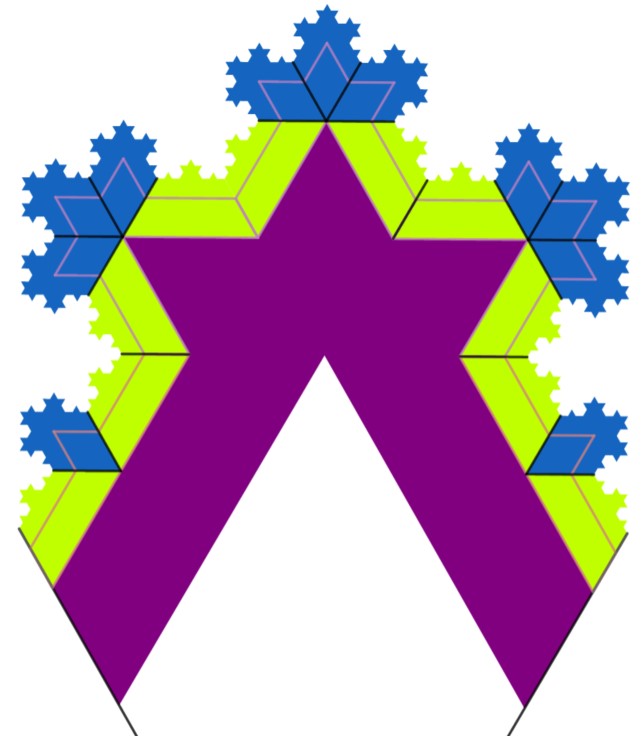}
\caption{Blue and yellow regions for the third step of construction.}\label{KON7}
\end{center}
\end{figure}\\
In every region we repeat the construction according to the color of the region. Let us observe that in the next generation every lime region \textcolor{chn}{has} 3 subregions (lime,blue,lime) and every blue region \textcolor{chn}{has} 5 subregions (lime,blue,blue,blue,lime). Since the vertices of neighboring polygons coincide we can repeat the construction inductively.\\
Let $K_n$ be $n$-th approximation of von Koch's snowflake. By $G_n$ we denote the set covered by polygons from $n$-th generation of the construction, where $G_0$ is just \textcolor{chn}{the} six pointed star in middle of von Koch's snowflake.\\ The polygons on the $n$-th step of the construction almost cover the set $K_{n+2}$ (except a narrow strip next to the boundary). Observe that we always perform the same construction on lime and blue regions. However the regions on n-th step are \textcolor{chn}{the} scaled copies of regions from the second step with a scale $\frac{1}{3^{n-2}}$ for $n\geq 2$. Therefore there exists a constant $C>0$ such that 
\[
\{x\in K_{n+2} :  \operatorname{dist}(x,\partial K_{n+2}) <\frac{C}{3^n}\} \subset \bigcup_{j=0}^{n} G_n.
\]
Since for $k\in \mathbb{N}$ we have $K_n\subset K_{n+k}$ and for any $x\in K_n$ the sequence $\operatorname{dist}(x,\partial K_{n+j})$ is non-increasing. We get
\[
K_n\subset\overline{\bigcup_{n=0}^{\infty} G_n}.
\]          
Therefore 
\[
\overline{\Omega_K}=\overline{\bigcup_{n=0}^{\infty} K_n}=\overline{\bigcup_{n=0}^{\infty} G_n},
\]
where $\Omega_K$ is von Koch's snowflake. Obviously
\[
\bigcup_{n=0}^{\infty} G_n\subset \Omega_K.
\]
Hence the family of polygons we constructed covers von Koch's snowflake. Other properties of \textcolor{chn}{the} Whitney covering follow easily from the construction.\\

\textbf{Acknowledgments} We would like to thank Anna Kamont for valuable comments and suggestions. We would like to thank anonymous reviewers for stimulating comments. During the work on the article, a by-product, a handmade fractal carpet, was created. As an artistic object, it was exhibited at the Bridges Linz 2019 conference. Its manufacturing was supported by the Copernicus Science Center in Warsaw. We thank them for their support of our artistic endeavors. \\

 \textbf{Funding} This research was partially supported by the National Science Centre, Poland, and Austrian Science Foundation FWF joint CEUS programme. National Science Centre project no. 2020/02/Y/ST1/00072 and FWF project no. I5231.\\

\bibliographystyle{plain}
\bibliography{bibliografia}

\begin{thebibliography}{10}

\bibitem{bonk2018}
Mario Bonk, Eero Saksman, and Tom\'{a}s Soto.
\newblock Triebel-{L}izorkin spaces on metric spaces via hyperbolic fillings.
\newblock {\em Indiana Univ. Math. J.}, 67(4):1625--1663, 2018.

\bibitem{MR2882877}
Alexander Brudnyi and Yuri Brudnyi.
\newblock {\em Methods of geometric analysis in extension and trace problems.
  {V}olume 1}, volume 102 of {\em Monographs in Mathematics}.
\newblock Birkh\"{a}user/Springer Basel AG, Basel, 2012.

\bibitem{MR1359964}
S.~Buckley and P.~Koskela.
\newblock Sobolev-{P}oincar\'{e} implies {J}ohn.
\newblock {\em Mathematical Research Letters}, 2(5):577--593, 1995.

\bibitem{carlfrac}
Lennart Carleson.
\newblock On the support of harmonic measure for sets of {C}antor type.
\newblock {\em Ann. Acad. Sci. Fenn. Ser. A I Math.}, 10:113--123, 1985.

\bibitem{MR0132389}
Z.~Ciesielski.
\newblock On the isomorphisms of the spaces {$H_{\alpha }$} and {$m$}.
\newblock {\em Bulletin de l'Acad\'{e}mie Polonaise des Sciences. S\'{e}rie des
  Sciences Math\'{e}matiques, Astronomiques et Physiques}, 8:217--222, 1960.

\bibitem{Derezin}
Micha{\l} Derezinski.
\newblock Isomorphic properties of function space bv on simply connected planar
  sets.
\newblock Master's thesis, Uniwersity of Warsaw, 2013.

\bibitem{DNW}
Micha{\l} Derezinski, Fedor Nazarov, and Micha{\l} Wojciechowski.
\newblock Isomorphic properties of function space bv on simply connected planar
  sets.
\newblock in preparation.

\bibitem{MR0102739}
Emilio Gagliardo.
\newblock Caratterizzazioni delle tracce sulla frontiera relative ad alcune
  classi di funzioni in {$n$} variabili.
\newblock {\em Rendiconti del Seminario Matematico della Universit\`a di
  Padova. The Mathematical Journal of the University of Padova}, 27:284--305,
  1957.

\bibitem{Gmeineder2019}
Franz Gmeineder, Bogdan Rai\c{t}\u{a}, and Jean Van~Schaftingen.
\newblock On limiting trace inequalities for vectorial differential operators.
\newblock {\em Indiana Univ. Math. J.}, 70(5):2133--2176, 2021.

\bibitem{MR1428124}
Piotr Haj{\l}asz and Olli Martio.
\newblock Traces of {S}obolev functions on fractal type sets and
  characterization of extension domains.
\newblock {\em Journal of Functional Analysis}, 143(1):221--246, 1997.

\bibitem{jonsson}
Alf Jonsson and Hans Wallin.
\newblock Function spaces on subsets of {${\bf R}^n$}.
\newblock {\em Math. Rep.}, 2(1):xiv+221, 1984.

\bibitem{Koskela2017}
Pekka Koskela, Tapio Rajala, and Yi~Ru-Ya Zhang.
\newblock A density problem for sobolev spaces on gromov hyperbolic domains.
\newblock {\em Nonlinear Analysis: Theory, Methods {\&} Applications},
  154:189--209, 2017.

\bibitem{MR3519964}
Pekka Koskela and Yi~Ru-Ya Zhang.
\newblock A density problem for {S}obolev spaces on planar domains.
\newblock {\em Archive for Rational Mechanics and Analysis}, 222(1):1--14,
  2016.

\bibitem{Lahti}
Panu Lahti, Xining Li, and Zhuang Wang.
\newblock Traces of {N}ewton-{S}obolev, {H}ajlasz-{S}obolev, and {BV} functions
  on metric spaces.
\newblock {\em Ann. Sc. Norm. Super. Pisa Cl. Sci. (5)}, 22(3):1353--1383,
  2021.

\bibitem{MR0500056}
Joram Lindenstrauss and Lior Tzafriri.
\newblock {\em Classical {B}anach spaces. {I}}.
\newblock Springer-Verlag, Berlin-New York, 1977.
\newblock Sequence spaces, Ergebnisse der Mathematik und ihrer Grenzgebiete,
  Vol. 92.

\bibitem{Malysb}
Lukáš Malý.
\newblock Trace and extension theorems for sobolev-type functions in metric
  spaces, 2017.

\bibitem{Meyer}
Yves Meyer.
\newblock {\em Wavelets and operators}, volume~37 of {\em Cambridge Studies in
  Advanced Mathematics}.
\newblock Cambridge University Press, Cambridge, 1992.
\newblock Translated from the 1990 French original by D. H. Salinger.

\bibitem{MR0552011}
Jaak Peetre.
\newblock A counterexample connected with {G}agliardo's trace theorem.
\newblock {\em Commentationes Mathematicae. Special Issue}, 2:277--282, 1979.
\newblock Special issue dedicated to W\l adys\l aw Orlicz on the occasion of
  his seventy-fifth birthday.

\bibitem{pelcz}
Aleksander Pe{\l}czy\'{n}ski.
\newblock Projections in certain {B}anach spaces.
\newblock {\em Studia Math.}, 19:209--228, 1960.

\bibitem{MR1979184}
Aleksander Pe{\l}czy\'{n}ski and Micha{\l} Wojciechowski.
\newblock Spaces of functions with bounded variation and {S}obolev spaces
  without local unconditional structure.
\newblock {\em Journal f\"{u}r die Reine und Angewandte Mathematik. [Crelle's
  Journal]}, 558:109--157, 2003.

\bibitem{RW}
Maria Roginskaya and Michal Wojciechowski.
\newblock Bounded approximation property for {S}obolev spaces on
  simply-connected planar domains.
\newblock {\em https://arxiv.org/abs/1401.7131}.

\bibitem{Weaver1999}
Nik Weaver.
\newblock {\em Lipschitz algebras}.
\newblock World Scientific Publishing Co., Inc., River Edge, NJ, 1999.

\end{thebibliography}
Krystian Kazaniecki\\ \vspace{0.5cm}
krystian.kazaniecki@jku.at\\ 
Michał Wojciechowski\\
miwoj@impan.pl

\end{document}